\documentclass[pre,aps,preprint,amsmath,amssymb,showkeys]{revtex4}
\usepackage{graphicx}
\usepackage{epsfig}
\usepackage{epstopdf}
\include{graphics}
\DeclareMathOperator{\sech}{sech}

\usepackage{amsthm}
\usepackage{multirow}

\usepackage{algorithm}
\usepackage{algpseudocode}

\newtheorem{thm}{Theorem}[section]
\newtheorem{lm}[thm]{Lemma}

\newtheorem{pro}[thm]{Proposition}
\numberwithin{equation}{section}

\renewcommand\Re{\operatorname{Re}}
\renewcommand\Im{\operatorname{Im}}

\newtheorem{App}[thm]{}

\def\xx{{\mathrm x}}

\def\T{\mathcal{T}}

\begin{document}

\title{A Finite Difference Scheme for (2+1)D Cubic-Quintic Nonlinear Schr\"odinger Equations  with Nonlinear Damping}

\author{Anh-Ha Le$^{1,4}$}

\author{Toan T. Huynh$^{2}$}

\author{Quan M. Nguyen$^{3,4}$}
\email[Corresponding author. E-mail address: ]{quannm@hcmiu.edu.vn (Q.M. Nguyen)}

\affiliation{$^{1}$Faculty of Mathematics and Computer Science, University of Science, Ho Chi Minh City, Vietnam}

\affiliation{$^{2}$Department of Mathematics, University of Medicine and Pharmacy at Ho Chi Minh City, Ho Chi Minh City, Vietnam}

\affiliation{$^{3}$Department of Mathematics, International University, Ho Chi Minh City, Vietnam}

\affiliation{$^{4}$Vietnam National University, Ho Chi Minh City, Vietnam}

\date{\today}

\begin{abstract}
Solitons of the purely cubic nonlinear Schr\"odinger equation in a space dimension of $n \geq 2$ suffer critical and supercritical collapses. These solitons can be stabilized in a cubic-quintic nonlinear medium. In this paper, we analyze the Crank-Nicolson finite difference scheme for the (2+1)D cubic-quintic nonlinear Schr\"odinger equation with cubic damping. We show that both the discrete solution, in the discrete $L^2$-norm, and discrete energy are bounded. By using appropriate settings and estimations, the existence and the uniqueness of the numerical solution are proved. In addition, the error estimations are established in terms of second order for both space and time in discrete $L^2$-norm and $H^1$-norm. Numerical simulations for the (2+1)D cubic-quintic nonlinear Schr\"odinger equation with cubic damping are conducted to validate the convergence.
\end{abstract}


\maketitle


\section{Introduction}\label{Intro}
In this paper, we establish and analyze the Crank-Nicolson finite difference scheme for the (2+1)-dimensional ((2+1)D) cubic-quintic nonlinear Schr\"odinger (CQNLS) equation with nonlinear damping. The nonlinear Schr\"odinger equation (NLS) is one of the most ubiquitous nonlinear equations in mathematics and physics. It describes the dynamics of solitons in diverse physical models such as nonlinear optics, nanophotonics, plasma physics, and Bose-Einstein condensates \cite{Ablowitz_2011}. The CQNLS possesses a flat-top soliton solution, which appears in various nonlinear dispersive media such as nonlinear optics, fluid physics, plasma physics, and Bose-Einstein condensates \cite{Kevrekidis_2008, Hora_2004}. It has been discovered that the 2D flat-top soliton, which is described by the (2+1)D CQNLS equation, can exist in cubic-quintic nonlinear media \cite{Malomed2012, PRL2013, Skarka_PRE1997}.

Since the 1980s, many numerical schemes for NLS-type equations have been investigated extensively, e.g, the finite difference methods in Refs. \cite{Akrivis1993, Bao2003, SS84, Sun2010, Weideman1997}, the Ablowitz-Ladik scheme in \cite{Taha1984}, Crank-Nicholson Galerkin finite element in \cite{Henning2017, Tou91, Wang2014}, the pseudo-spectral split-step method in \cite{Lubich2008, Weideman1986}, the relaxation method in \cite{Besse2004}, the time-integrator scheme in \cite{Besse2021, Henning2023}. In particular, an error estimation of the relaxation finite difference scheme for the (1+1)D NLS equation has been proposed recently in \cite{Zouraris2023}.  An extension of the numerical method for solving (1+1)D NLS equations to $(n+1)$D NLS equations, for $n \geq 2$, is very promising and challenging. For the $(n+1)$D NLS equation in a high dimensional space with $n \geq 2$, there is a rich literature on the numerical treatments presented in \cite{Bao2013a, Besse2021, Bao2013b, Henning2023, Henning2021, Lubich2008, Thalhammer2008, Wang2014} and references therein. In \cite{Wang2014}, the author studied linearized Crank-Nicolson Galerkin FEMs, which is an Adams-Bashforth-type linearization of the Crank-
Nicolson method  for a generalized
nonlinear Schr\"odinger equation in two-dimensional and three-dimensional space. Additionally, the time-splitting spectral (TSSP) scheme is also a common method for numerically solving the damped ($n$+1)D NLS-type equation \cite{Bao2013a, Bao2003B}. The TSSP is very efficient and performs very well for high regularity solutions.  However, the performance of the TSSP method can drop dramatically when the solution has low regularity \cite{Knoller2019}. It has been shown that both TSSP and CNFD schemes hold the mass conservation \cite{Bao2013a}. However, the TSSP scheme does not hold the conservation of energy \cite{Bao2013a}. Recently, the numerical integration in time of nonlinear
Schr\"odinger equations using the method of preserving energy was introduced in \cite{Besse2021}. Also, an optimal H$^1$-error estimates for Crank-Nicolson approximations to the nonlinear Schr\"odinger equation was proposed in \cite{Henning2021}. Additionally, the  uniform $L^{\infty}$-bounds for the continuous Galerkin discretization of the Gross–Pitaevskii equation with rotation recently has investigated in \cite{Henning2023}. Among these numerical methods, there have been proposed some finite difference schemes for the (2+1)D NLS equations with wave operator or with cubic nonlinearity (also known as Kerr nonlinearity) and a potential \cite{Bao2013a, Bao2012, Bao2013b}. More specifically, the {\it unperturbed} (2+1)D NLS and its specific version, the Gross-Pitaevskii equation with a potential in higher dimensional spaces describing 2D and 3D solitons in Bose-Einstein condensates, has been studied in Refs. \cite{Bao2013a, Bao2013b}. In \cite{Bao2013b}, the authors analyzed finite difference methods for the unperturbed Gross-Pitaevskii equation, i.e. the (2+1)D cubic NLS, with an angular momentum rotation term in two and three dimensions and obtained the optimal convergence rate. In this work, the numerical schemes for the perturbed (2+1)D cubic NLS were based on the methods of conservative Crank-Nicolson finite difference (CNFD) and semi-implicit finite difference (SIFD).

As aforementioned, the 2D soliton can be stabilized in a modified nonlinearity medium such as cubic-quintic and that they can be {\it perturbed} by some nonlinear processes such as cubic damping due to two-photon absorption. However, in spite of the importance of the cubic-quintic nonlinearity to stabilize 2D soliton, which is used recently in 2D materials, a rigorous study on numerical schemes of (2+1)D NLS equation with cubic-quintic nonlinearity and a damping perturbation (without an external potential) is {\it still lacking}. In fact, the CNFD scheme for the unperturbed (2+1)D cubic NLS equation with a potential has been proposed and studied in \cite{Bao2013b}. However, for the (2+1)D CQNLS equation with nonlinear damping, the Hamiltonian is not conserved when using the CNFD. Therefore, it is more difficult to show the existence and uniqueness of the numerical solution. Unlike the CNFD scheme for the (2+1)D cubic NLS equation with a potential, the perturbed (2+1)D CQNLS equation studied in this paper includes the interplay between the cubic-quintic nonlinearity and the nonlinear damping in a two-dimensional spatial space. As a result, the condition for stability and convergence can be changed. To date, there is no proof for the existence, uniqueness and error estimation for the (2+1)D CQNLS equation with nonlinear damping. Therefore, it certainly requires a rigorous analysis and a technical treatment for uniqueness and convergence of the numerical solution for this equation.

In this work, we aim to fill this gap. More specifically, we propose a Crank-Nicolson finite difference scheme for the (2+1)D CQNLS with a perturbation of cubic damping:
\begin{align}
	\left\{
	\begin{array}{cc}\label{eq:NLS}
		&i\frac{\partial u}{\partial t}  + \Delta u + \lambda u|u|^2 - \nu |u|^4u + i\varepsilon u |u|^2 = 0~~\forall \xx \in \Omega= (a,b)\times(c,d), ~t\in(0,T),\\
		&u(\xx,0) = u_0(\xx),~~\forall \xx \in \Omega,\\
		&u(\xx,t) = 0,~~~~\forall \xx \in \Gamma = \partial \Omega, ~ t\in [0,T],
	\end{array}
	\right.
\end{align}
where $\lambda$ is a real constant, $\varepsilon \geq 0$ and constant $\nu > 0$. The nonlinear-damping $i\varepsilon u |u|^2$ arises in many physical systems. In nonlinear optics, the nonlinear damping is also called nonlinear loss, which is due to multi-photon absorption. In BEC, a quintic nonlinear damping term corresponds to losses from condensate due to three-body inelastic recombinations \cite{Bao2004}. Additionally, the complex-Ginzburg-Landau equation with a nonlinear-damping term have been used to model Poiseuille flow, Rayleigh-Bernard convection, Taylor-Couette flow, and superconductivity \cite{{Fibich2011}}. For $\varepsilon = 0$, we will show that the proposed numerical scheme maintains the discrete mass conservation and the discrete energy conservation.

It is well-known that for a linear wave equation, Crank-Nicolson finite difference scheme has the convergence rate of the second order in both time step and space step. However, for a nonlinear wave equation, one needs to propose a specific treatment for the existence of solution and the convergence. As a result, the rate of convergence can be different. In the current paper, we will present the boundedness of discrete solutions, the existence, the uniqueness of discrete solution, and the convergence rate in detail. We close this section by summarizing the structure of the remainder of the paper as follows. In section \ref{CNFD_subsect}, we present the discrete CNFD scheme for the evolution equation. In section \ref{Error and result}, we establish the main error estimate results for the numerical scheme, including the boundedness of discrete solutions, the existence and uniqueness of solutions, and the analysis of convergence rates. Section \ref{Num} is reserved for the results of numerical experiments and section \ref{Con} is for the conclusion. In \ref{Appendix}, we show the consistency of the CNFD scheme.

\section{Crank-Nicolson finite difference scheme}
\label{CNFD_subsect}

The Crank-Nicolson finite difference scheme is a widely used numerical method for approximating partial differential equations. It offers several advantages, including improved stability and accuracy compared to explicit schemes. In this section, we provide a detailed explanation of the formulation of the Crank-Nicolson scheme in time and finite difference in space for equation \eqref{eq:NLS}.

On the space domain $\Omega$, we use a uniform Cartesian grid consisting of grid points $(x_{j},y_{k})$ on $\Omega$, where $$x_j = a+j\Delta x = a+j\frac{b-a}{J} \text{ and } y_k = c+k\Delta y =c+k\frac{d-c}{K},~\forall (j,k)\in \overline{0,J}\times \overline{0,K}.$$
Let us define the mesh size $h := h_{max}= \max\{\Delta x,\Delta y\}$ and let $h_{min} = \min\{\Delta x, \Delta y\}$. On the time domain $[0, T]$, we define the time step $\tau$ as 
\begin{align*}
	t_n=n\tau,~ n = 0,1,\cdots,N.
\end{align*}
Let $U_{j,k}^n$ represent an approximation of $u(x_j,y_k,t_n)$. To conveniently write the scheme and prove the convergence, the following discrete operators are defined
\begin{align*}
	&\delta^+_t u^n_{j,k} = \frac{u^{n+1}_{j,k}-u^n_{j,k}}{\tau},  \delta_x u_{j,k}^n= \frac{u^n_{j+1,k}-u^n_{j-1,k}}{2\Delta x},  \delta_x^+ u_{j,k}^n= \frac{u^n_{j+1,k}-u^n_{j,k}}{\Delta x}, \delta_x^- u_{j,k}^n= \frac{u^n_{j,k}-u^n_{j-1,k}}{\Delta x},\\
	&\delta_t u^n_{j,k} = \frac{u^{n+1}_{j,k}-u^{n-1}_{j,k}}{2\tau},  \delta_y u_{j,k}^n= \frac{u^n_{j,k+1}-u^n_{j,k-1}}{2\Delta y},  \delta_y^+ u_{j,k}^n= \frac{u^n_{j,k+1}-u^n_{j,k}}{\Delta y}, \delta_y^- u_{j,k}^n= \frac{u^n_{j,k}-u^n_{j,k-1}}{\Delta y},\\
	&\delta^2_{x} u_{j,k}^n  = \frac{u_{j+1,k}^n-2u_{j,k}^n+u_{j-1,k}^n}{(\Delta x)^2}, ~~\delta^2_{y} u_{j,k}^n  = \frac{u_{j,k+1}^n-2u_{j,k}^n+u_{j,k-1}^n}{(\Delta y)^2},\\
	&\delta u^n_{j,k} = (\delta_x u_{j,k}^n,\delta_y u_{j,k}^n),~~\delta^+ u^n_{j,k} = (\delta_x^+ u_{j,k}^n,\delta_y^+ u_{j,k}^n),~~\delta^2 u^n_{j,k} =  \delta^2_x u^n_{j,k} + \delta^2_y u^n_{j,k}.
\end{align*}
Setting
\begin{align*}
	&\T_{JK} = \big\{(j,k)| j =1,2,\cdots,J-1,~k=1,2,\cdots,K-1\big\},\\
	&\T^0_{JK} = \big\{(j,k)| j =0,1,\cdots,J,~k=0,1,\cdots,K\big\},\\
	&X_{JK}  = \left\{u = (u_{ik})_{(j,k)\in \T^0_{JK}}|  u_{0,k} = u_{J,k} = u_{j,0}=u_{j,K} = 0,~\forall (j,k)\in \T^0_{JK}\right\},
\end{align*}
equipped with inner products and norms  defined by
\begin{align*}
	&(u,v)_h = \Delta x\Delta y\sum_{j=0}^{J-1}\sum_{k=0}^{K-1}u_{j,k}\bar{v}_{j,k},~\|u\|^2_{2,h} = ( u,u)_h,~\|u\|^p_{p,h} = \Delta x \Delta y\sum_{j=0}^{J-1}\sum_{k=0}^{K-1}|u_{j,k}|^p,\\
	& ~\langle u,v\rangle_h = \Delta x\Delta y\sum_{j=1}^{J-1}\sum_{k=1}^{K-1}u_{j,k}\bar{v}_{j,k},~\|u\|_{\infty,h} = \sup_{(j,k)\in\T^0_{JK}}|u_{j,k}|,\\
	&(\delta^+u,\delta^+v)_h = (\delta^+_xu,\delta^+_xv)_h + (\delta^+_yu,\delta^+_yv)_h,~\|\delta^+u\|^2_{2,h}= \|\delta_x^+ u\|^2_{2,h} + \|\delta_y^+ u\|^2_{2,h}.
\end{align*}
The NLS in \eqref{eq:NLS} can be discretized  by the CNFD scheme as follows
\begin{align}
	\left\{
	\begin{array}{c}
		i\delta_t^+U^n_{j,k}+\delta^2 U_{j,k}^{n+1/2} + \lambda \psi_1(U^{n+1}_{j,k},U^n_{j,k})-\nu\psi_2(U^{n+1}_{j,k},U^n_{j,k}) +i\varepsilon \varphi(U^{n+1}_{j,k},U^n_{j,k}) =0, \\
		U^0_{j,k} = u_0(x_j,y_k);~ U^{n+1}\in X_{JK},
	\end{array}
	\right.\label{CNFD}
\end{align}
where $U^{n+1/2}_{j,k} = \frac{1}{2}(U^{n+1}_{j,k}+U^n_{j,k})$,  for all $(j,k) \in \T_{JK}$ and $n=\overline{0,N-1}$. Additionally, in \eqref{CNFD}, the functions $\psi_1$, $\psi_2$, and $\varphi$ are defined as follows
\begin{align}
	\psi_1(z,w)=\frac{F_1(|z|^2)-F_1(|w|^2)}{|z|^2-|w|^2}\frac{z+w}{2},~~\psi_2(z,w)=\frac{F_2(|z|^2)-F_2(|w|^2)}{|z|^2-|w|^2}\frac{z+w}{2},\label{df:psi1psi}
\end{align}
or 
\begin{align*}
	\psi_1(z,w)&=\frac{z+w}{2}\int^1_0f_1(|w|^2+t(|z|^2-|w|^2))dt=\frac{|z|^2+|w|^2}{2}\frac{z+w}{2},\\
	\psi_2(z,w)&=\frac{z+w}{2}\int^1_0f_2(|w|^2+t(|z|^2-|w|^2))dt=\frac{|z|^4+|z|^2|w|^2+|w|^4}{3}\frac{z+w}{2},
\end{align*}
with 
\begin{align}
	f_1(s)=s, f_2(s)= s^2, F_1(\rho)=\int_0^\rho f_1(s)ds=\frac{s^2}{2}, F_2(\rho)=\int_0^\rho f_2(s)ds=\frac{s^3}{3},\label{df_f1f2_F1F2}
\end{align}
and
\begin{align}
	\varphi(z,w) = \left|\frac{z+w}{2}\right|^2\frac{z+w}{2}.\label{df:varphi}
\end{align}
Furthermore, the discrete energy is defined by
\begin{align}
	E_h(U^n) &= \|\delta^+U^{n}\|^2_{2,h} -\frac{\lambda}{2}  \|U^n\|^4_{4,h}+\frac{\nu}{3}\|U^n\|^6_{6,h}.\label{df:E_h(U^n)}
\end{align}
These settings above play a crucial role in our analysis.

\section{Error Estimates and main results}
\label{Error and result}
In this section we establish the error estimates and the boundedness of the discrete solution and discrete energy. We then show the existence and the uniqueness of the numerical solution. Finally, we prove the second order convergence which is the main result of the paper. 

\subsection{Error Estimates}
\label{Error}
First, we establish some error estimates which are useful for estimating the discrete energy and for proving the uniqueness solution. From the definitions of  $(\cdot,\cdot)_h$ and $\langle\cdot,\cdot\rangle_h$, for $u,v\in X_{JK}$, it can be shown for the following discrete equalities \cite{Bao2013b}:
\begin{align}
	\langle \delta_x u,v \rangle_h =- \langle  u, \delta_x v \rangle_h,~~~~\langle \delta^2_x u,v \rangle_h = - ( \delta_x^+ u,\delta^+_x v)_h\label{discretegreen_x},\\
	\langle \delta_y u,v \rangle_h =- \langle u, \delta_y v \rangle_h,~~~~\langle \delta^2_y u,v \rangle_h = - ( \delta_y^+ u,\delta^+_y v)_h\label{discretegreen_y}.
\end{align}
Next, we present two estimates for any $u\in X_{JK}$ in the following Lemma. 
\begin{lm}\label{lm:est_normu2p}
	For any $u\in X_{JK}$, we have the following estimations for $u$:
	\begin{align}
		\|u\|^4_{4,h}\leq &\frac{1}{2}\|u\|_{2,h}^2\|\delta^+u\|^2_{2,h},\label{ineq:poincare1}\\
		\|u\|^{8}_{8,h}\leq  & 2\|u\|^{6}_{6,h}\|\delta^+ u\|^{2}_{2,h}.\label{est_normu2p_2}
	\end{align}
\end{lm}
\begin{proof}
	For $u\in X_{JK}$ and $p\in\mathbb{N}^+$, one has
	\begin{align*}
		|u_{j,k}|^{p}&=\bigg|\sum_{l=0}^{j-1}\big[|u_{l+1,k}|^{p}-|u_{l,k}|^{p}\big]\bigg| = \bigg|\sum_{l=0}^{j-1}(|u_{l+1,k}|-|u_{l,k}|)\bigg[\sum_{q=0}^{p-1}|u_{l+1,k}|^{p-1-q}|u_{l,k}|^{q}\bigg]\bigg|\\
		&\leq   \Delta x\sum_{l=0}^{j-1}|\delta_x^+u_{l,k}|\bigg[\sum_{q=0}^{p-1}|u_{l+1,k}|^{p-1-q}|u_{l,k}|^{q}\bigg].\nonumber
	\end{align*}
	Similarly, 
	\begin{align*}
		|u_{j,k}|^{p}
		&\leq   \Delta x\sum_{l=j}^{J-1}|\delta_x^+u_{l,k}|\bigg[\sum_{q=0}^{p-1}|u_{l+1,k}|^{p-1-q}|u_{l,k}|^{q}\bigg].\nonumber
	\end{align*}
	Summing over two above inequalities, there holds
	\begin{align*}
		2|u_{j,k}|^{p} &\leq   \Delta x\sum_{l=0}^{J-1}|\delta_x^+u_{l,k}|\bigg[\sum_{q=0}^{p-1}|u_{l+1,k}|^{p-1-q}|u_{l,k}|^{q}\bigg].
	\end{align*}
	Applying the Cauchy-Schwarz inequality, it leads to 
	\begin{align*}
		2|u_{j,k}|^{p} & \leq \Delta x\sqrt{\sum_{l=0}^{J-1}|\delta_x^+u_{l,k}|^2}\sqrt{\sum_{l=0}^{J-1}\bigg[\sum_{q=0}^{p-1}|u_{l+1,k}|^{p-1-q}|u_{l,k}|^{q}\bigg]^2}.
	\end{align*}
	Using the Cauchy inequality, we obtain
	\begin{align*}
		2|u_{j,k}|^{p}  & \leq \sqrt{p}\Delta x\sqrt{\sum_{l=0}^{J-1}|\delta_x^+u_{l,k}|^2}\sqrt{\sum_{l=0}^{J-1}\bigg[\sum_{q=0}^{p-1}|u_{l+1,k}|^{2(p-1-q)}|u_{l,k}|^{2q}\bigg]}\nonumber\\
		& \leq \frac{p}{\sqrt{2}}\Delta x\sqrt{\sum_{l=0}^{J-1}|\delta_x^+u_{l,k}|^2}\sqrt{\sum_{l=0}^{J-1}\bigg[|u_{l+1,k}|^{2(p-1)}+|u_{l,k}|^{2(p-1)}\bigg]}\nonumber\\
		& \leq p\Delta x\sqrt{\sum_{l=0}^{J-1}|\delta_x^+u_{l,k}|^2}\sqrt{\sum_{l=0}^{J-1}|u_{l,k}|^{2(p-1)} }.
	\end{align*}
	That is,
	\begin{align}
		|u_{j,k}|^{p} 
		& \leq \frac{p}{2}\Delta x\sqrt{\sum_{l=0}^{J-1}|\delta_x^+u_{l,k}|^2}\sqrt{\sum_{l=0}^{J-1}|u_{l,k}|^{2(p-1)} }.\label{ineq:u^px_2}
	\end{align}
	Similarly, one can obtain
	\begin{align}
		|u_{j,k}|^p \leq \frac{p}{2}\Delta y\sqrt{\sum_{s=0}^{K-1}|\delta_y^+u_{j,s}|^2}\sqrt{\sum_{s=0}^{K-1}|u_{j,s}|^{2(p-1)}}.\label{ineq:u^py_2}
	\end{align}
	Combining \eqref{ineq:u^px_2} and \eqref{ineq:u^py_2} and using the Cauchy inequality, it implies
	\begin{align*}
		\|u\|^{2p}_{2p,h}&=\Delta x\Delta y\sum_{j=0,k=0}^{J-1,K-1}|u_{j,k}|^{2p} = \Delta x\Delta y\sum_{j=0,k=0}^{J-1,K-1}|u_{j,k}|^{p}\cdot |u_{j,k}|^{p}\nonumber\\
		&\leq \frac{p^2}{4}(\Delta x\Delta y)^2\sum_{j=0,k=0}^{J-1,K-1}\left[\sqrt{\sum_{l=0}^{J-1}|\delta_x^+u_{l,k}|^2}\sqrt{\sum_{l=0}^{J-1}|u_{l,k}|^{2(p-1)}} \right]\nonumber\\ 
		&\times\left[\sqrt{\sum_{s=0}^{K-1}|\delta_y^+u_{j,s}|^2}\sqrt{\sum_{s=0}^{K-1}|u_{j,s}|^{2(p-1)}}\right].\nonumber\\
	\end{align*} 
	The terms $\displaystyle \sqrt{\sum_{l=0}^{J-1}|\delta_x^+u_{l,k}|^2}\sqrt{\sum_{l=0}^{J-1}|u_{l,k}|^{2(p-1)}}$ and $\displaystyle \sqrt{\sum_{s=0}^{K-1}|\delta_y^+u_{j,s}|^2}\sqrt{\sum_{s=0}^{K-1}|u_{j,s}|^{2(p-1)}}$  only depend on $k$ and $j$, respectively. Thus, there holds
	\begin{align*}
		\|u\|^{2p}_{2p,h}&\leq \frac{p^2}{4}(\Delta x\Delta y)^2 \sum_{k=0}^{K-1}\left[\sqrt{\sum_{l=0}^{J-1}|\delta_x^+u_{l,k}|^2}\sqrt{\sum_{l=0}^{J-1}|u_{l,k}|^{2(p-1)}}\right] \nonumber\\
		& \times\sum_{j=0}^{J-1}\left[\sqrt{\sum_{s=0}^{K-1}|\delta_y^+u_{j,s}|^2}\sqrt{\sum_{s=0}^{K-1}|u_{j,s}|^{2(p-1)}}\right].\nonumber
	\end{align*}
	Using the Cauchy-Schwarz inequality, it yields 
	\begin{align*}
		\|u\|^{2p}_{2p,h} &\leq \frac{p^2}{4}(\Delta x\Delta y)^2\sqrt{\sum_{k=0}^{K-1}\sum_{l=0}^{J-1}|\delta_x^+u_{l,k}|^2}\sqrt{\sum_{l=0}^{J-1}\sum_{k=0}^{K-1}|u_{l,k}|^{2(p-1)}}\nonumber\\
		&\times\sqrt{\sum_{j=0}^{J-1}\sum_{s=0}^{K-1}|\delta_y^+u_{j,s}|^2}\sqrt{\sum_{j=0}^{J-1}\sum_{s=0}^{K-1}|u_{j,s}|^{2(p-1)}}\nonumber\\
		&\leq \frac{p^2}{8}(\Delta x\Delta y)^2\bigg[ \sum_{j=0,k=0}^{J-1,K-1}|\delta_x^+u_{j,k}|^2\sum_{j=0,k=0}^{J-1,K-1}|u_{j,k}|^{2(p-1)}\\
		&+\sum_{j=0,k=0}^{J-1,K-1}|\delta_y^+u_{j,k}|^2\sum_{j=0,k=0}^{J-1,K-1}|u_{j,k}|^{2(p-1)}\bigg] \\
		&= \frac{p^2}{8}\|u\|^{2(p-1)}_{2(p-1),h} \|\delta^+ u\|^2_{2,h}.
	\end{align*}
	Therefore, it arrives at
	\begin{align*}
		\|u\|^{2p}_{2p,h} \leq \frac{p^2}{8}\|u\|^{2(p-1)}_{2(p-1),h}\|\delta^+ u\|^2_{2,h}.
	\end{align*}
	With $p=2$ and $p=4$, one obtains \eqref{ineq:poincare1} and \eqref{est_normu2p_2}, respectively.\\
 This completes the proof of Lemma \ref{lm:est_normu2p} for the estimation of any $u\in X_{JK}$. 
\end{proof}

\subsection{Boundedness of discrete solution and energy}
In this subsection, we consider the discrete norm $\|\cdot\|_{2,h}$ of the solution and the  discrete energy in \eqref{df:E_h(U^n)}. We demonstrate the boundedness of the discrete solution in the discrete $L^2$-norm. For simplicity, we assume $\Delta x = \Delta y = h$. 
\begin{lm}[Boundedness of the discrete solution in the discrete $L^2$-norm]\label{lm:discreteConservativeMass}
	Let $\{U^n\}_{n=0}^{N}$ be a solution to the CNFD scheme (\ref{CNFD}), there holds for $n=1,\cdots,N$,
	\begin{align}
		\|U^n\|^2_{2,h} \leq \|U^0\|^2_{2,h}.\label{discreteConservativeMass}
	\end{align}
\end{lm}
\begin{proof}
	Taking in the first term on the left hand side of  scheme \eqref{CNFD} the inner product $\langle\cdot,\cdot\rangle_h$  with $U^{n+1/2}$, we get
	\begin{align}
		\!\!\!\!i\langle\delta_t^+U^n,U^{n+1/2}\rangle_h &= \frac{i}{2\tau}\langle U^{n+1}-U^n,U^{n+1}+U^n\rangle_h \nonumber\\
		&=i\frac{\|U^{n+1}\|^2_{2,h}-\|U^n\|^2_{2,h}}{2\tau} + \frac{i}{2\tau}(\langle U^{n+1},U^n\rangle_h-\langle U^{n},U^{n+1}\rangle_h).\label{eq:product:delta_tUn,Un+12}
	\end{align}
	Since
	$$
	\Re(\langle U^{n+1},U^n\rangle_h-\langle U^{n},U^{n+1}\rangle_h) = 0,
	$$
	one has
	\begin{align*}
		\Im\left(\frac{i}{2\tau}(\langle U^{n+1},U^n\rangle_h-\langle U^{n},U^{n+1}\rangle_h)\right) = 0.
	\end{align*}
	By taking the imaginary part of equation \eqref{eq:product:delta_tUn,Un+12}, it yields
	\begin{align}\label{impartdt}
		\Im(i\langle\delta_t^+U^n,U^{n+1/2}\rangle_h) =\frac{\|U^{n+1}\|^2_{2,h}-\|U^n\|^2_{2,h}}{2\tau}.
	\end{align}
	Using (\ref{discretegreen_x}) and (\ref{discretegreen_y}) and noting that $U^{n+1/2}\in X_{JK}$, we get
	\begin{align}
		\langle\delta^2_x U^{n+1/2},U^{n+1/2}\rangle_h = -(\delta^+_x U^{n+1/2},\delta^+_xU^{n+1/2})_h,\label{impartdiffusion_x}\\
		\langle\delta^2_y U^{n+1/2},U^{n+1/2}\rangle_h = -(\delta^+_y U^{n+1/2},\delta^+_y U^{n+1/2})_h.\label{impartdiffusion_y}
	\end{align}
	Adding \eqref{impartdiffusion_x} and \eqref{impartdiffusion_y} side by side and using the definitions of $\delta^2$ and $\|\delta^+\cdot\|_{2,h}$, there holds
	\begin{align}
		\langle\delta^2 U^{n+1/2},U^{n+1/2}\rangle_h = -\|\delta^+ U^{n+1/2}\|^2_{2,h}.\label{impartdiffusion}
	\end{align}
	Taking the inner product $\langle\cdot,\cdot\rangle_h$ of the third term, the fourth term, and the final term   in \eqref{CNFD} with $U^{n+1/2}$ and using the definitions of $\psi_1,\psi_2$ in \eqref{df:psi1psi} and $\varphi$ in \eqref{df:varphi}, one obtains
	\begin{align}
		\lambda\langle\psi_1(U^{n+1},U^n),U^{n+1/2}\rangle_h &= \lambda\left\langle \frac{F_1(|U^{n+1}|^2)-F_2(|U^n|^2)}{|U^{n+1}|^2-|U^n|^2},|U^{n+1/2}|^2 \right\rangle_h,\label{impart_cubicterm}\\
		\nu\langle\psi_2(U^{n+1},U^n),U^{n+1/2}\rangle_h &= \nu\left\langle \frac{F_2(|U^{n+1}|^2)-F_2(|U^n|^2)}{|U^{n+1}|^2-|U^n|^2},|U^{n+1/2}|^2 \right\rangle_h,\label{impart_qinticterm}\\
		i\varepsilon\langle\varphi(U^{n+1},U^n),U^{n+1/2}\rangle_h &= i\varepsilon\left\langle |U^{n+1/2}|^2,|U^{n+1/2}|^2 \right\rangle_h = \varepsilon\|U^{n+1/2}\|^4_{4,h}.\label{impart_cubiclossterm}
	\end{align}
	Totally, taking in (\ref{CNFD}) the inner product $\langle\cdot,\cdot\rangle_h$ with $U^{n+1/2}$, using (\ref{impartdt}), 
 (\ref{impartdiffusion}), (\ref{impart_cubicterm}),  \eqref{impart_qinticterm}, and  \eqref{impart_cubiclossterm}, and then taking the imaginary parts, it yields
	\begin{align*}
		\frac{\|U^{n+1}\|^2_{2,h}- \|U^{n}\|^2_{2,h}}{2\tau} + \varepsilon\|U^{n+1/2}\|^4_{4,h}=0~~\forall n=0,1,\cdots,N-1.
	\end{align*}
	Summing the last equations over $n=0,1,\cdots,m-1$ with $m\in \overline{1,N}$, we get
	\begin{align}
		\frac{\|U^m\|^2_{2,h}-\|U^0\|^2_{2,h}}{2\tau} +\varepsilon\sum_{n=0}^{m-1}\|U^{n+1/2}\|^4_{4,h} = 0.\label{eq:sum_L4Un}
	\end{align}
	It then arrives at
	\begin{align*}
		\|U^{m}\|^2_{2,h} \leq \|U^{0}\|^2_{2,h}~~\forall m=1,2,\cdots,N.
	\end{align*}
	Thus, the estimate \eqref{discreteConservativeMass} for the boundedness of the discrete solution is proved.
\end{proof}
To obtain the boundedness of the discrete energy, we first show the following two lemmas. 
\begin{lm}\label{lm:bound_gradient_byEnergy}
	For all $U^m\in X_{JK}$, it holds that
	\begin{align}\label{lm:est_H01Un}
		\|\delta^+ U^m\|^2_{2,h} + \frac{\nu}{6}\|U^m\|^6_{6,h}  \leq E_h(U^m) + \frac{3\lambda^2}{8\nu}\|U^m\|^2_{2,h}.
	\end{align}
\end{lm}
\begin{proof} The proof is split into two cases.\\ 
	{Case I:  $\lambda \leq 0$}. From the definition of the discrete energy in \eqref{df:E_h(U^n)}, one obtains
	\begin{align*}
		E_h(U^m) &= \|\delta^+ U^{m}\|^2_{2,h}  -\frac{\lambda}{2}\|U^m\|^4_{4,h} +\frac{\nu}{3}\|U^m\|^6_{6,h} \geq \|\delta^+U^{m}|^2_{2,h}+\frac{\nu}{6}\|U^m\|^6_{6,h}.
	\end{align*}
	Then
	\begin{align}
		\|\delta^+U^m\|^2_{2,h} + \frac{\nu}{6}\|U^m\|^6_{6,h}  \leq E_h(U^m) + \frac{3\lambda^2}{8\nu}\|U^m\|^2_{2,h}. \label{ineq:est_H01Un_1}
	\end{align}
	{Case II: $\lambda > 0$}. From the definition of the discrete energy in \eqref{df:E_h(U^n)}, one gets
	\begin{align*}
		E_h(U^m) &= \|\delta^+U^{m}\|^2_{2,h}  -\frac{\lambda}{2}\|U^m\|^4_{4,h} + \frac{\nu}{3}\|U^m\|^6_{6,h}\\
		& \geq \|\delta^+U^{m}\|^2_{2,h} - \frac{\lambda}{2}\|U^m\|^3_{6,h}\|U^m\|_{2,h} + \frac{\nu}{3}\|U^m\|^6_{6,h}\\
		& = \|\delta^+U^{m}\|^2_{2,h} - \frac{3\lambda^2}{8\nu}\|U^m\|^2_{2,h} + \bigg(\sqrt{\frac{\nu}{6}}\|U^m\|^3_{6,h}-\frac{\lambda}{4}\sqrt{\frac{6}{\nu}}\|U^m\|_{2,h}\bigg)^2 + \frac{\nu}{6}\|U^m\|^6_{6,h}\\
		& \geq \|\delta^+U^{m}\|^2_{2,h} - \frac{3\lambda^2}{8\nu}\|U^m\|^2_{2,h} + \frac{\nu}{6}\|U^m\|^6_{6,h}.
	\end{align*}
	Then
	\begin{align}
		\|\delta^+U^m\|^2_{2,h}+\frac{\nu}{6}\|U^m\|^6_{6,h}  \leq E(U^m) + \frac{3\lambda^2}{8\nu}\|U^m\|^2_{2,h}. \label{ineq:est_H01Un_2}
	\end{align}
	The inequalities \eqref{ineq:est_H01Un_1} and \eqref{ineq:est_H01Un_2} complete the proof.
\end{proof}
\begin{lm} 
	Let $\{U^n\}_{n=0}^N$ be a solution to the CNFD scheme \eqref{CNFD}.  Then,
	\begin{align}
		\Im\big(\big\langle U^{n+1},{U^n}|U^{n+1/2}|^2\big\rangle_h\big)\leq 
		\tau\lambda\left\langle \frac{F_1(|U^{n+1}|^2)-F_1(|U^n|^2)}{|U^{n+1}|^2-|U^n|^2},|U^{n+1/2}|^4\right\rangle_h \nonumber\\ - \tau\nu\left\langle \frac{F_2(|U^{n+1}|^2)-F_2(|U^n|^2)}{|U^{n+1}|^2-|U^n|^2},|U^{n+1/2}|^4\right\rangle_h .\label{est:Im_Un+1_Un_U_n2}
	\end{align}
\end{lm}
\begin{proof}
	Note that
	\begin{align*}
		i\delta_t^+U^n_{j,k}\overline{U^{n+1/2}_{j,k}} &= \frac{i}{2\tau}\big(|U^{n+1}_{j,k}|^2-|U^n_{j,k}|^2\big) + \frac{i}{2\tau}\big(U^{n+1}_{j,k}\overline{U^n_{j,k}} - U^{n}_{j,k}\overline{U^{n+1}_{j,k}}\big)\\
		&=\frac{i}{2\tau}\big(|U^{n+1}_{j,k}|^2-|U^n_{j,k}|^2\big) - \frac{1}{\tau}\Im\big(U^{n+1}_{j,k}\overline{U^n_{j,k}} \big).
	\end{align*}
	Multiplying the previous equality by $|U^{n+1/2}_{j,k}|^2$ and then summing them over $j=1,\cdots,J-1$ and $k=1,\cdots,K-1$, one arrives at 
	\begin{align}
		i\big\langle\delta_t^+U^n,{U^{n+1/2}}|U^{n+1/2}|^2\big\rangle_h = \frac{i}{2\tau}\big\langle |U^{n+1}|^2-|U^n|^2,|U^{n+1/2}|^2\big\rangle_h\nonumber\\
		- \frac{1}{\tau}\Im\big(\big\langle U^{n+1},{U^n}|U^{n+1/2}|^2\big\rangle_h\big).\label{equa:delta_t_U^2}
	\end{align}
	From the inequalities (\ref{discretegreen_x}) and (\ref{discretegreen_y}), it implies
	\begin{align}
		\big\langle\delta^2_x U^{n+1/2},U^{n+1/2}|U^{n+1/2}|^2\big\rangle_h = -\big(\delta^+_x U^{n+1/2},\delta^+_x\big(U^{n+1/2}|U^{n+1/2}|^2\big)\big)_h,\label{equa:green_delta_x_U^2}\\
		\big\langle\delta^2_y U^{n+1/2},U^{n+1/2}|U^{n+1/2}|^2\big\rangle_h = -\big(\delta^+_y U^{n+1/2},\delta^+_y\big(U^{n+1/2}|U^{n+1/2}|^2\big)\big)_h.\label{equa:green_delta_y_U^2}
	\end{align}
	From the definition of the inner product $(\cdot,\cdot)_h$ and the discrete gradient $\delta_x^+$, we have
	\begin{align}
		L_1&:=\big(\delta^+_x U^{n+1/2},\delta^+_x\big(U^{n+1/2}|U^{n+1/2}|^2\big)\big)_h \nonumber\\
		&= \Delta x \Delta y\sum_{j=0,k=0}^{J-1,K-1}\frac{U^{n+1/2}_{j+1,k}-U^{n+1/2}_{j,k}}{\Delta x}\frac{\overline{U^{n+1/2}_{j+1,k}}|U^{n+1/2}_{j+1,k}|^2-\overline{U^{n+1/2}_{j,k}}|U^{n+1/2}_{j,k}|^2}{\Delta x}.\label{def:L1}
	\end{align}
	Moreover, 
	\begin{align*}
		\overline{U^{n+1/2}_{j+1,k}}|U^{n+1/2}_{j+1,k}|^2-\overline{U^{n+1/2}_{j,k}}|U^{n+1/2}_{j,k}|^2 = \big(\overline{U^{n+1/2}_{j+1,k}} - \overline{U^{n+1/2}_{j,k}}\big)\frac{|U^{n+1/2}_{j+1,k}|^2+|U^{n+1/2}_{j,k}|^2}{2}\nonumber\\ + \frac{\overline{U^{n+1/2}_{j+1,k}} + \overline{U^{n+1/2}_{j,k}}}{2}\big(|U^{n+1/2}_{j+1,k}|^2-|U^{n+1/2}_{j,k}|^2\big).
	\end{align*}
	Substituting this equation into \eqref{def:L1}, one gets
	\begin{align*}
		L_1= \Delta x \Delta y\sum_{j=0,k=0}^{J-1,K-1}\frac{U^{n+1/2}_{j+1,k}-U^{n+1/2}_{j,k}}{\Delta x}\frac{\overline{U^{n+1/2}_{j+1,k}} - \overline{U^{n+1/2}_{j,k}}}{\Delta x}\frac{|U^{n+1/2}_{j+1,k}|^2+|U^{n+1/2}_{j,k}|^2}{2}\\
		+\Delta x \Delta y\sum_{j=0,k=0}^{J-1,K-1}\frac{U^{n+1/2}_{j+1,k}-U^{n+1/2}_{j,k}}{\Delta x}\frac{\overline{U^{n+1/2}_{j+1,k}}+\overline{U^{n+1/2}_{j,k}}}{2}\frac{|U^{n+1/2}_{j+1,k}|^2-|U^{n+1/2}_{j,k}|^2}{\Delta x}.
	\end{align*}
	Taking the real part of $L_1$ in the last equation, one obtains
	\begin{align*}
		\Re(L_1)& = \frac{\Delta x \Delta y}{2}\sum_{j=0,k=0}^{J-1,K-1}\bigg|\frac{U^{n+1/2}_{j+1,k}-U^{n+1/2}_{j,k}}{\Delta x}\bigg|^2\left(|U^{n+1/2}_{j+1,k}|^2+|U^{n+1/2}_{j,k}|^2\right)\\
		&+\frac{\Delta x \Delta y}{2} \sum_{j=0,k=0}^{J-1,K-1}\bigg|\frac{|U^{n+1/2}_{j+1,k}|^2-|U^{n+1/2}_{j,k}|^2}{\Delta x}\bigg|^2.
	\end{align*}
	We note that $\Re(L_1)\geq 0$. From \eqref{equa:green_delta_x_U^2} and the definition of $L_1$ in \eqref{def:L1}, it implies
	\begin{align}
		\Re\big(\big\langle\delta^2_x U^{n+1/2},U^{n+1/2}|U^{n+1/2}|^2\big\rangle_h\big) = - \Re(L_1) \leq 0\label{inequa:Re_delta_x_U^2}.
	\end{align}
	By using \eqref{equa:green_delta_y_U^2} and implementing the similar calculations for respect $y$, it arrives at
	\begin{align}
		\Re\big(\big\langle\delta^2_y U^{n+1/2},U^{n+1/2}|U^{n+1/2}|^2\big\rangle_h\big) \leq 0 \label{inequa:Re_delta_y_U^2}.
	\end{align}
	Adding two inequalities \eqref{inequa:Re_delta_x_U^2} and \eqref{inequa:Re_delta_y_U^2} side by side and recalling the definition of the operator $\delta^2$, we have
	\begin{align}
		\Re\big(\big\langle\delta^2 U^{n+1/2},U^{n+1/2}|U^{n+1/2}|^2\big\rangle_h\big) \leq 0. \label{inequa:Re_delta_U^2}
	\end{align}
	Taking the inner product $\langle\cdot,\cdot\rangle_h$ of the third term, the fourth term, and the final term   in \eqref{CNFD} with $U^{n+1/2}|U^{n+1/2}|^2$ and using the definitions of $\psi_1,\psi_2$ and $\varphi$, one can obtain the following equalities  
	\begin{align}
		\lambda\langle\psi_1(U^{n+1},U^n),U^{n+1/2}|U^{n+1/2}|^2\rangle_h &= \lambda\left\langle \frac{F_1(|U^{n+1}|^2)-F_2(|U^n|^2)}{|U^{n+1}|^2-|U^n|^2},|U^{n+1/2}|^4 \right\rangle_h,\label{equa:cubictem}\\
		\nu\langle\psi_2(U^{n+1},U^n),U^{n+1/2}|U^{n+1/2}|^2\rangle_h &= \nu\left\langle \frac{F_2(|U^{n+1}|^2)-F_2(|U^n|^2)}{|U^{n+1}|^2-|U^n|^2},|U^{n+1/2}|^4 \right\rangle_h,\label{equa:qinticterm}\\
		i\varepsilon\langle\varphi(U^{n+1},U^n),U^{n+1/2}|U^{n+1/2}|^2\rangle_h &= i\varepsilon\left\langle |U^{n+1/2}|^2,|U^{n+1/2}|^4 \right\rangle_h.\label{equa:cubiclossterm}
	\end{align}
We take in (\ref{CNFD}) the inner product $\langle\cdot,\cdot\rangle_h$ with $U^{n+1/2}|U^{n+1/2}|^2$ and then take the real parts of the results. By using \eqref{equa:delta_t_U^2},\eqref{inequa:Re_delta_U^2} \eqref{equa:cubictem}, \eqref{equa:qinticterm}, and \eqref{equa:cubiclossterm}, it yields
	\begin{align*}
		- \frac{1}{\tau}\Im\big(\langle U^{n+1},{U^n}|U^{n+1/2}|^2\rangle_h\big)  + \lambda\left\langle \frac{F_1(|U^{n+1}|^2)-F_1(|U^n|^2)}{|U^{n+1}|^2-|U^n|^2},|U^{n+1/2}|^4\right\rangle_h \\ - \nu\left\langle \frac{F_2(|U^{n+1}|^2)-F_2(|U^n|^2)}{|U^{n+1}|^2-|U^n|^2}|,|U^{n+1/2}|^4\right\rangle_h  \geq 0.
	\end{align*}
	That is,
	\begin{align*}
		\Im\big(\big\langle U^{n+1},{U^n}|U^{n+1/2}|^2\big\rangle_h\big)\leq 
		\tau\lambda\left\langle \frac{F_1(|U^{n+1}|^2)-F_1(|U^n|^2)}{|U^{n+1}|^2-|U^n|^2},|U^{n+1/2}|^4\right\rangle_h \nonumber\\ - \tau\nu\left\langle \frac{F_2(|U^{n+1}|^2)-F_2(|U^n|^2)}{|U^{n+1}|^2-|U^n|^2},|U^{n+1/2}|^4\right\rangle_h.
	\end{align*}
	This estimation completes the proof of the lemma.
\end{proof}

\begin{lm}[Boundedness of discrete energy]\label{lm:EUnbounded}
	Let $\{U^n\}_{n=0}^N$ be a solution to the CNFD scheme \eqref{CNFD}. We have
	\begin{align}
		E_h(U^n) \leq E_h(U^0)+\frac{\lambda^2}{4\nu} \|U^0\|^2_{2,h},~~\forall n=0,1,\cdots,N.\label{est_EhUn_EhU0}
	\end{align}
\end{lm}
\begin{proof}
	From the definition of the discrete difference operator $\delta_t^+$, there holds
	\begin{align}\label{realpartdt}
		i\big\langle\delta_t^+U^n,U^{n+1}-U^n\big\rangle_h = i\frac{\|U^{n+1}-U^n\|^2_{2,h}}{\tau}.
	\end{align}
	Using (\ref{discretegreen_x}),  (\ref{discretegreen_y}) and the  definition of the discrete norm $\|\delta^+\cdot\|_{2,h}$, one obtains
	\begin{align}
		&\big\langle\delta^2_x U^{n+1/2},U^{n+1}-U^n\big\rangle_h + \big\langle\delta^2_y U^{n+1/2},U^{n+1}-U^n\big\rangle_h\nonumber\\
		&= -(\delta^+_x U^{n+1/2},\delta^+_x(U^{n+1}-U^n))_h -(\delta^+_y U^{n+1/2},\delta^+_y (U^{n+1}-U^n))_h\nonumber\\
		&= -\frac{1}{2}\|\delta^+ U^{n+1}\|^2_{2,h} + \frac{1}{2}\|\delta^+U^n\|^2_{2,h} + \frac{1}{2}(\delta^+_x U^{n+1},\delta^+_xU^n)_h - \frac{1}{2}\big(\delta^+_x U^{n},\delta^+_xU^{n+1}\big)_h \nonumber\\ 
		&+ \frac{1}{2}\big(\delta^+_y U^{n+1},\delta^+_yU^n\big)_h - \frac{1}{2}\big(\delta^+_y U^{n},\delta^+_y U^{n+1}\big)_h.\label{realpartdiffusion}
	\end{align}
	We note that 
	\begin{align*}
		\Re\left(\frac{1}{2}(\delta^+_x U^{n+1},\delta^+_xU^n)_h - \frac{1}{2}\big(\delta^+_x U^{n},\delta^+_xU^{n+1}\big)_h \right) = 0,\\
		\Re\left(\frac{1}{2}\big(\delta^+_y U^{n+1},\delta^+_yU^n\big)_h - \frac{1}{2}\big(\delta^+_y U^{n},\delta^+_y U^{n+1}\big)_h\right) = 0.
	\end{align*}
	By taking the real part of \eqref{realpartdiffusion}, we have
	\begin{align}
		\!\!\!\!\Re\big(\big\langle\delta^2_x U^{n+1/2},U^{n+1}-U^n\big\rangle_h+ \big\langle\delta^2_y U^{n+1/2},U^{n+1}-U^n\big\rangle_h\big) = -\frac{1}{2}\|\delta^+ U^{n+1}\|^2_{2,h} + \frac{1}{2}\|\delta^+U^n\|^2_{2,h}. \label{realpartdiffusion1}
	\end{align}
	From the third term in \eqref{CNFD} and the definitions of $\varphi_1$ and $F_1$, we have
	\begin{align}
		\lambda\langle\psi_1(U^{n+1},U^n),U^{n+1}-U^n\rangle_h &= \lambda \left\langle \frac{F_1(|U^{n+1}|^2)-F_1(|U^n|^2)}{|U^{n+1}|^2-|U^n|^2}U^{n+1/2},U^{n+1}-U^n \right\rangle_h\nonumber\\
		&= \frac{\lambda}{4} \left\langle (|U^{n+1}|^2+|U^n|^2)(U^{n+1}+U^n),U^{n+1}-U^n \right\rangle_h\nonumber\\
		&= \frac{\lambda}{4}\left<|U^{n+1}|^2+|U^n|^2,|U^{n+1}|^2-|U^n|^2\right>_h\nonumber \\
		&+ \frac{\lambda}{4}\left<|U^{n+1}|^2+|U^n|^2,\overline{U^n}U^{n+1}-\overline{U^{n+1}}U^{n}\right>_h.\label{realpart_cubictem}
	\end{align}
	 From \eqref{realpart_cubictem} and noting that \begin{align*}
		\Re\left(\left<|U^{n+1}|^2+|U^n|^2,\overline{U^n}U^{n+1}-\overline{U^{n+1}}U^{n}\right>_h\right) = 0,
	\end{align*} there holds
	\begin{align}
		\Re(\lambda\langle\psi_1(U^{n+1},U^n),U^{n+1}-U^n\rangle_h) = \frac{\lambda}{4}\|U^{n+1}\|^4_{4,h} - \frac{\lambda}{4}\|U^n\|^4_{4,h}.\label{realpart_cubictem1}
	\end{align}
	By performing the similar calculations for the fourth term in \eqref{CNFD}, from the definitions of $\varphi_2$ and $F_2$, one has
	\begin{align}
		\nu\langle\psi_2(U^{n+1},U^n),U^{n+1}-U^n\rangle_h &= \nu \left\langle \frac{F_2(|U^{n+1}|^2)-F_2(|U^n|^2)}{|U^{n+1}|^2-|U^n|^2}U^{n+1/2},U^{n+1}-U^n \right\rangle_h\nonumber\\
		&= \frac{\nu}{6} \left\langle (|U^{n+1}|^4+|U^{n+1}|^2|U^n|^2+|U^n|^4)(U^{n+1}+U^n),U^{n+1}-U^n \right\rangle_h\nonumber\\
		&= \frac{\nu}{6}\left<|U^{n+1}|^4 + |U^{n+1}|^2|U^n|^2+|U^n|^4,|U^{n+1}|^2-|U^n|^2\right>_h \nonumber\\ 
		&+\frac{\nu}{6}\left<|U^{n+1}|^4 + |U^{n+1}|^2|U^n|^2+|U^n|^4,\overline{U^{n}}U^{n+1}- \overline{U^{n+1}}U^{n} \right>_h.\label{realpart_quintictem}
	\end{align}
	Also, from \eqref{realpart_quintictem} and noting that
	\begin{align*}
		\Re\left(\left<|U^{n+1}|^4 + |U^{n+1}|^2|U^n|^2+|U^n|^4,\overline{U^{n}}U^{n+1}- \overline{U^{n+1}}U^{n} \right>_h\right) = 0,
	\end{align*}
	 there holds
	\begin{align}
		\Re(\nu\langle\psi_2(U^{n+1},U^n),U^{n+1}-U^n\rangle_h) = \frac{\nu}{6}\|U^{n+1}\|^6_{6,h} - \frac{\nu}{6}\|U^n\|^6_{6,h}.\label{realpart_quintictem1}
	\end{align}
	Additionally, for the cubic loss term, one can obtain
	\begin{align}
		i\varepsilon\big\langle \varphi(U^{n+1},U^n),U^{n+1}-U^n \big\rangle_h =i\varepsilon\big\langle U^{n+1/2}|U^{n+1/2}|^2,U^{n+1}-U^n \big\rangle_h\nonumber\\
		= \frac{i\varepsilon}{2}\big\langle |U^{n+1/2}|^2,|U^{n+1}|^2-|U^n|^2\big\rangle_h
		+\varepsilon\Im\big(\big\langle U^{n+1},U^n|U^{n+1/2}|^2\big\rangle_h\big).\label{loss_cubic}
	\end{align}	
We take in (\ref{CNFD}) the inner product $\langle\cdot,\cdot\rangle_h$ with $U^{n+1}-U^n$ and then take the real parts of the results. By using (\ref{realpartdt}), (\ref{realpartdiffusion1}), \eqref{realpart_cubictem1}, \eqref{realpart_quintictem1}, and \eqref{loss_cubic}, it yields
	\begin{align*}
		-\frac{1}{2}\|\delta^+ U^{n+1}\|^2_{2,h} + \frac{1}{2}\|\delta^+ U^{n}\|^2_{2,h}  +\frac{\lambda}{4}\|U^{n+1}\|^4_{4,h} -\frac{\lambda}{4}\|U^{n}\|^4_{4,h} +\frac{\nu}{6}\|U^{n+1}\|^6_{6,h}- \frac{\nu}{6}\|U^{n}\|^6_{6,h} \\
		+ \varepsilon\Im\big(\big\langle U^{n+1},U^n|U^{n+1/2}|^2\big\rangle_h\big)=0.
	\end{align*}
	That is,
	\begin{align*}
		\big(\|\delta^+ U^{n+1}\|^2_{2,h}  -\frac{\lambda}{2}\|U^{n+1}\|^4_{4,h}+\frac{\nu}{3}\|U^{n+1}\|^6_{6,h}\big) 
		- \big(\|\delta^+ U^{n}\|^2_{2,h}  -\frac{\lambda}{2}\|U^{n}\|^4_{4,h}+\frac{\nu}{3}\|U^{n}\|^6_{6,h}\big) \\
		=2\varepsilon\Im\big(\big\langle U^{n+1},U^n|U^{n+1/2}|^2\big\rangle_h\big).
	\end{align*}
	From  the definition of $E_h(U^n)$ in \eqref{df:E_h(U^n)}, the previous equation can be written as
	\begin{align}
		E_h(U^{n+1})-E_h(U^n) 
		=2\varepsilon\Im\big(\big\langle U^{n+1},U^n|U^{n+1/2}|^2\big\rangle_h\big).\label{equa:E_h(U^{n+1})-E_h(U^{n})}
	\end{align}
	Applying the estimation for the term $\Im\big(\big\langle U^{n+1},U^n|U^{n+1/2}|^2\big\rangle_h\big)$ in \eqref{est:Im_Un+1_Un_U_n2} with $\varepsilon\geq 0$, the last equation becomes 
	\begin{align}
		E_h(U^{n+1})-E_h(U^{n}) \leq 
		2\tau\lambda\left\langle\frac{F_1(|U^{n+1}|^2)-        F_1(|U^n|^2)}{|U^{n+1}|^2-|U^n|^2}|U^{n+1/2}|^2,|U^{n+1/2}|^2\right\rangle_h \nonumber \\- 2\tau\nu\left\langle \frac{F_2(|U^{n+1}|^2)-F_2(|U^n|^2)}{|U^{n+1}|^2-|U^n|^2}|U^{n+1/2}|^2,|U^{n+1/2}|^2\right\rangle_h .\label{est:E_h(U^{n+1})-E_h(U^{n})}
	\end{align}
	Now we consider the two cases for $\lambda$. \\
	$\bullet$ If $\lambda \leq 0$. From the definitions of $F_1$ and $F_2$, it yields
	\begin{align}        
		\left\langle\frac{F_1(|U^{n+1}|^2)-F_1(|U^n|^2)}       {|U^{n+1}|^2-                                         U^n|^2},|U^{n+1/2}|^4\right\rangle_h  =  \frac{1}           {2}\left< (|U^{n+1}|^2+|U^n|^2)|U^{n+1/2}|^2,|U^{n+1/2}|^2\right>_h\geq 0, \label{equa:explain_F1}\\
		\!\!\!\!\!\!\left\langle\frac{F_2(|U^{n+1}|^2)-F_2(|U^n|^2)}{|U^{n+1}|^2-|U^n|^2} ,|U^{n+1/2}|^4\right\rangle_h =\frac{1}{3}\left<|U^{n+1}|^4+|U^{n+1}|^2|U^n|^2+|U^n|^4,|U^{n+1/2}|^4\right> \geq 0.\label{equa:explain_F2}
	\end{align}
	Combining the inequalities \eqref{est:E_h(U^{n+1})-E_h(U^{n})}, \eqref{equa:explain_F1},  \eqref{equa:explain_F2}, $\nu>0$, and $\lambda \leq 0$,  one obtains 
	\begin{align}
		E_h(U^{n+1})- E_h(U^{n})\leq 0. \label{est:E_h(U^n)_1}
	\end{align}
	$\bullet$ If $\lambda > 0$. 
	Applying H\"older's inequality for the right hand side of \eqref{equa:explain_F1}, there holds
	\begin{align}
		\left\langle\frac{F_1(|U^{n+1}|^2)-F_1(|U^n|^2)}{|U^{n+1}|^2-|U^n|^2},|U^{n+1/2}|^4\right\rangle_h \leq \frac{1}{2}\|\big(|U^{n+1}|^2+|U^n|^2\big)|U^{n+1/2}|^2\|_{2,h}\|U^{n+1/2}\|^2_{4,h}.\label{est_F1}
	\end{align}
	Plugging the inequality $\displaystyle\frac{1}{3}(|U^{n+1}|^4+|U^{n+1}|^2|U^n|^2+|U^n|^4) \geq \frac{1}{4}(|U^{n+1}|^2+|U^n|^2)^2$ into the equality \eqref{equa:explain_F2}, we get
	\begin{align}
		\left\langle\frac{F_2(|U^{n+1}|^2)-F_2(|U^n|^2)}{|U^{n+1}|^2-|U^n|^2} ,|U^{n+1/2}|^4\right\rangle_h &\geq \frac{1}{4}\|\big(|U^{n+1}|^2+|U^n|^2\big)|U^{n+1/2}|^2\|^2_{2,h}.\label{est_F2}
	\end{align}
	From the inequalities \eqref{est_F1}, \eqref{est_F2}, and \eqref{est:E_h(U^{n+1})-E_h(U^{n})}, one can arrive at
	\begin{align}
		E_h(U^{n+1}) -E_h(U^n)& \leq \tau\varepsilon\lambda\|(U^{n+1}|^2+|U^n|^2)|U^{n+1/2}|^2\|_{2,h}\|U^{n+1/2}\|^2_{4,h} \nonumber\\ 
		& -\frac{\tau\varepsilon\nu}{2}\|(|U^{n+1}|^2+|U^n|^2)|U^{n+1/2}|^2\|^2_{2,h} \nonumber\\
		&\leq -\varepsilon\tau\bigg(\sqrt{\frac{\nu}{2}}\|(|U^{n+1}|^2+|U^n|^2)|U^{n+1/2}|^2\|_{2,h} -\frac{\lambda}{\sqrt{2\nu}}\|U^{n+1/2}\|^2_{4,h}\bigg) ^2 \nonumber\\
		&+ \frac{\lambda^2\varepsilon\tau}{2\nu} \|U^{n+1/2}\|^4_{4,h} \leq \frac{\lambda^2\varepsilon\tau}{2\nu} \|U^{n+1/2}\|^4_{4,h}.\label{est:E_h(U^n)_11}
	\end{align} 
	Finally, from the observations for two cases of $\lambda$ and from the inequalities \eqref{est:E_h(U^n)_1} and \eqref{est:E_h(U^n)_11}, it yields
	\begin{align}
		E_h(U^{n+1}) -E_h(U^n) \leq \frac{\lambda^2\varepsilon\tau}{2\nu} \|U^{n+1/2}\|^4_{4,h},~~\forall n=0,1,\cdots,N-1.\label{est:E_h(U^n)_2}
	\end{align}
	Summing the inequalities \eqref{est:E_h(U^n)_2} over $n=0,1,\cdots,m-1$, with $m\in\overline{1,N}$, there holds
	\begin{align}
		E_h(U^{m})-E_h(U^0)\leq \frac{\lambda^2\varepsilon\tau}{2\nu}\sum_{n=0}^{m-1}\|U^{n+1/2}\|^4_{4,h}.\label{est:E_h(U^n)_3}
	\end{align}
	On the other hand, from the equation \eqref{eq:sum_L4Un}, one obtains
	\begin{align}
		\tau\varepsilon \sum_{n=0}^{m-1}\|U^{n+1/2}\|^4_{4,h}\leq \frac{1}2\|U^0\|^2_{2,h}. \label{est:sum_L4Un}
	\end{align}
	Combining the two estimations  \eqref{est:E_h(U^n)_3} and \eqref{est:sum_L4Un}, it yields
	\begin{align*}
		E_h(U^{m})-E_h(U^0) \leq \frac{\lambda^2}{4\nu}\|U^0\|^2_{2,h},~\forall m=1,\cdots,N.
	\end{align*}
	This proves the boundedness of the discrete energy of \eqref{est_EhUn_EhU0}.
\end{proof}

\begin{itemize}
    \item {\bf Remark 3.1} Note that when $\varepsilon=0$, from Eq. \eqref{est:Im_Un+1_Un_U_n2}, one can obtain  the discrete conservation of mass. Additionally, from Eq.  \eqref{equa:E_h(U^{n+1})-E_h(U^{n})}, one also gets $E_h(U^{n+1})=E_h(U^n)$ for all $n=0,\cdots,N-1$. It means that discrete energies $E_h(U^n)$ are conserved.
    
\end{itemize}    

\subsection{Existence and uniqueness solution}
The following Brouwer fixed point theorem \cite{Akrivis1991} is used to prove the existence of a solution of the CNFD scheme  in \eqref{CNFD}.
\begin{lm}\label{lm:existencesolution}
	Let $(H(\cdot,\cdot))$ be a finite dimensional inner product space, $\|\cdot\|$ be the associated norm, and $g: H\rightarrow H$ be continuous. Assume moreover that
	$$
	\exists \alpha > 0 ~\forall z \in H ~ \|z\| = \alpha,~ \Re(g(z),z) \geq 0.
	$$
	Then, there exists a $z^*\in H$ such that $g(z^*) = 0$ and $\|z^*\|\leq \alpha$.
\end{lm}
\begin{thm}[Existence and uniqueness solution]\label{thm:exist_unique_sol_CNFD}
	Assume $U^n \in  X_{JK}$ is a solution to the CNFD scheme \eqref{CNFD} and $\tau$ is small enough. There exists a unique solution $U^{n+1}$ to the CNFD discretization (\ref{CNFD}).
\end{thm}
\begin{proof}
	From (\ref{CNFD}), one obtains for all $(j,k)\in\T_{JK}$
	\begin{align*}
		U^{n+1/2}_{j,k} = U^n_{j,k} +\frac{i\tau}{2}\delta^2 U^{n+1/2}_{j,k} + \frac{i\tau\lambda}{2}\frac{|2U^{n+1/2}_{j,k}-U^n_{j,k}|^2+|U^n_{j,k}|^2}{2}U^{n+1/2}_{j,k}\\
		-\frac{i\tau\nu}{2}\frac{|2U^{n+1/2}_{j,k}-U^n_{j,k}|^4+|2U^{n+1/2}_{j,k}-U^n_{j,k}|^2|U^n_{j,k}|^2+|U^n_{j,k}|^4}{3}U^{n+1/2}_{j,k}-\frac{\tau\varepsilon}{2} |U^{n+1/2}_{j,k}|^2 U^{n+1/2}_{j,k}.
	\end{align*}
	We define the mapping $\Pi: X_{JK}\rightarrow X_{JK}$ as follows
	\begin{align}
		(\Pi(v))_{j,k} = v_{j,k}-U^n_{j,k}  -  \frac{i\tau}{2}\delta^2 v_{j,k} - \frac{i\tau\lambda}{2}\frac{|2v_{j,k}-U^n_{j,k}|^2+|U^n_{j,k}|^2}{2}v_{j,k}\nonumber\\+\frac{i\tau\nu}{2}\frac{|2v-U^n_{j,k}|^4+|2v_{j,k}-U^n_{j,k}|^2|U^n_{j,k}|^2+|U^n_{j,k}|^4}{2}v_{j,k} + \frac{\tau\varepsilon}{2} |v_{j,k}|^2  v_{j,k}.\label{def:Pi}
	\end{align}
	We note that the function $\Pi$ is continuous. Taking the mapping $\Pi(v)$  the inner product $\langle\cdot\rangle_h$ with $v$, there holds
	$$
	\Re\langle\Pi(v),v\rangle_h \geq \|v\|^2_{2,h}-\Re(\langle U^n,v\rangle_h).
	$$ 
    That is,
	\begin{align*}
		\Re\langle\Pi(v),v\rangle_h \geq \|v\|_{2,h}(\|v\|_{2,h}-\|U^n\|_{2,h}).
	\end{align*}
	By choosing  $\|v\|_{2,h} = \|U^n\|_{2,h}+1$, one then gets $\Re\langle\Pi(v),v\rangle_h > 0$. Hence from Lemma \ref{lm:existencesolution}, there exists a solution $U^{n+1/2}\in X_{JK}$.\\
	Now we prove the unique solution in the scheme \eqref{CNFD}. Let $v,w\in X_{JK}$ be solutions in the scheme \eqref{CNFD} such that $\Pi(v) = \Pi(w) = 0$, where $\Pi$ is defined in  \eqref{def:Pi}. Setting $\chi = v-w$, one obviously has
	\begin{align}\label{unique_CNFD}
		\chi_{j,k} =\frac{i\tau}{2}\delta^2 \chi_{j,k} + \frac{i\lambda\tau}{2}\phi_1(v_{j,k},w_{j,k})-\frac{i\nu\tau}{2}\phi_2(v_{j,k},w_{j,k}) +\frac{\tau\varepsilon}{2}\phi_3(v_{j,k},w_{j,k}),
	\end{align}
	where 
	\begin{align*}
		\phi_1(v_{j,k},w_{j,k}) &= \frac{|2v_{j,k}-U^n_{j,k}|^2+|U^n_{j,k}|^2}{2}v_{j,k}-\frac{|2w_{j,k}-U^n_{j,k}|^2+|U^n_{j,k}|^2}{2}w_{j,k},\\
		\phi_2(v_{j,k},w_{j,k})&=\frac{|2v_{j,k}-U^n_{j,k}|^4+|2v_{j,k}-U^n_{j,k}|^2|U^n_{j,k}|^2+|U^n_{j,k}|^4}{3}v_{j,k}\\ 
		&- \frac{|2w_{j,k}-U^n_{j,k}|^4+|2w_{j,k}-U^n_{j,k}|^2|U^n_{j,k}|^2+|U^n_{j,k}|^4}{3}w_{j,k},\\
		\phi_3(v_{j,k},w_{j,k}) &= |v_{j,k}|^2v_{j,k}-|w_{j,k}|^2w_{j,k}.
	\end{align*}
	Now we will estimate the term $\phi_1(v_{j,k},w_{j,k})$. Note that
	\begin{align*}
		\phi_1(v_{j,k},w_{j,k}) &= \frac{|2v_{j,k}-U^n_{j,k}|^2+|U^n_{j,k}|^2}{2}v_{j,k}-\frac{|2w_{j,k}-U^n_{j,k}|^2+|U^n_{j,k}|^2}{2}w_{j,k}\\
		&=\frac{|2v_{j,k}-U^n_{j,k}|^2+|U^n_{j,k}|^2}{2}\chi_{j,k} +\frac{|2v_{j,k}-U^n_{j,k}|^2-|2w_{j,k}-U^n_{j,k}|^2}{2}w_{j,k}.
	\end{align*}
	Therefore,
	\begin{align*}
		|\phi_1(v_{j,k},w_{j,k})|&\leq 4(|v_{j,k}|^2+|U^n_{j,k}|^2)|\chi_{j,k}| + (2|v_{j,k}|+2|w_{j,k}|+2|U^n_{j,k}|)|w_{j,k}||\chi_{j,k}|\\
		&\leq 5(|v_{j,k}|^2+|w_{j,k}|^2+|U^n_{j,k}|^2)|\chi_{j,k}|.
	\end{align*}
	Similarly, there holds
	\begin{align*}
		\phi_2(v_{j,k},w_{j,k})
		&=\frac{|2v_{j,k}-U^n_{j,k}|^4+|2v_{j,k}-U^n_{j,k}|^2|U^n_{j,k}|^2+|U^n_{j,k}|^4}{3}\chi_{j,k} \\+ &(|2v_{j,k}-U^n_{j,k}|^2-|2w_{j,k}-U^n_{j,k}|^2)\frac{|2v_{j,k}-U^n_{j,k}|^2+|2w_{j,k}-U^n_{j,k}|^2+|U^n_{j,k}|^2}{3}w_{j,k}.
	\end{align*}
	It implies
	\begin{align*}
		|\phi_2(v_{j,k},w_{j,k})|\leq 40(|v_{j,k}|^4+|w_{j,k}|^4+|U^n_{j,k}|^4)|\chi_{j,k}|.
	\end{align*}
	Similarly, one can obtain
	\begin{align*}
		|\phi_3(v_{j,k},w_{j,k})|&\leq |v_{j,k}|^2|\chi_{j,k}|+||v_{j,k}|^2-|w_{j,k}|^2||w_{j,k}|\\
		&\leq |v_{j,k}|^2|\chi_{j,k}| +(|v_{j,k}|+|w_{j,k}|)|\chi_{j,k}||w_{j,k}|\\
		&\leq |\chi_{j,k}|(|v_{j,k}|^2+|v_{j,k}||w_{j,k}| + |w_{j,k}|^2)\\
		&\leq \frac{3}{2}(|v_{j,k}|^2+ |w_{j,k}|^2)|\chi_{j,k}|.
	\end{align*}
	Taking in (\ref{unique_CNFD}) the inner product $\langle\cdot,\cdot\rangle_h$ with $\chi$, taking the real and imaginary parts, respectively, and using H\"older's inequality on the right-hand sides of the resulting identities, it yields
	\begin{align}
		\|\chi\|_{2,h}^2 &\leq \frac{5\tau|\lambda|}{2}(\|v\|^2_{4,h} +\|w\|^2_{4,h} + \|U^n\|^2_{4,h})\|\chi\|^2_{4,h}  +\frac{3\varepsilon\tau}{4}(\|v\|^2_{4,h}+\|w\|^2_{4,h}) \|\chi\|^2_{4,h} \nonumber\\ &+20\nu\tau(\|v\|^4_{8,h} +\|w\|^4_{8,h} + \|U^n\|^4_{8,h})\|\chi\|^2_{4,h}, \label{ineq:est_norml2chi}\\
		\|\delta^+\chi\|_{2,h}^2& \leq 5|\lambda|(\|v\|^2_{4,h} +\|w\|^2_{4,h} + \|U^n\|^2_{4,h})\|\chi\|^2_{4,h}  + \frac{3\varepsilon}{2}(\|v\|^2_{4,h}+\|w\|^2_{4,h}) \|\chi\|^2_{4,h} \nonumber\\ & +40\nu(\|v\|^4_{8,h} +\|w\|^4_{8,h} + \|U^n\|^4_{8,h})\|\chi\|^2_{4,h}.\label{ineq:est_normh1chi}
	\end{align}
	From  the inequality  \eqref{est_normu2p_2} in Lemma \ref{lm:est_normu2p}, Lemma \ref{lm:discreteConservativeMass}, Lemma \ref{lm:bound_gradient_byEnergy}, and Lemma \ref{lm:EUnbounded}, it implies, for $m=0,1,\cdots,N$, 
	\begin{align*}
		\|U^m\|^4_{8,h} &\leq \sqrt{2}\|U^m\|^3_{6,h}\|\delta^+ U^m\|_{2,h} \leq \frac{\sqrt{3}}{\sqrt{\nu}}
		\left(\|\delta^+U^m\|^2_{2,h} + \frac{\nu}{6}\|U^m\|^6_{6,h}\right)\\
		&\leq \frac{\sqrt{3}}{\sqrt{\nu}}\left( E(U^m)+\frac{3\lambda^2}{8\nu}\|U^m\|^2_{2,h}\right) \\
		&\leq  \frac{\sqrt{3}}{\sqrt{\nu}}\left(E_h(U^0) +\frac{5\lambda^2}{8\nu} \|U^0\|^2_{2,h}\right).
	\end{align*}
	Then $\|v\|^4_{8,h}, \|w\|^4_{8,h}, \|U^n\|^4_{8,h}$ are bounded. Applying the inequality \eqref{ineq:poincare1} in Lemma \ref{lm:est_normu2p}, we have 
	\begin{align*}
		\|U^m\|^4_{4,h}\leq \|\delta^+U^m\|^2_{2,h}\|U^m\|^2_{2,h}\leq \left(E_h(U^0) + \frac{5\lambda^2}{8\nu}\|U^0\|^2_{2,h}\right)\|U^0\|^2_{2,h}.
	\end{align*}
	From previous inequality, $\|v\|^2_{4,h}, \|w\|^2_{4,h}, \|U^n\|^2_{4,h}$ are bounded. The inequalities (\ref{ineq:est_norml2chi}), (\ref{ineq:est_normh1chi}) become
	\begin{align}
		\|\chi\|_{2,h}^2 &\leq c_1\tau\|\chi\|^2_{4,h}\label{ineq:est_norml2chi1},\\
		\|\delta^+\chi\|_{2,h}^2 &\leq c_2\|\chi\|^2_{4,h}.\label{ineq:est_normh1chi2}
	\end{align}
	Multiplying (\ref{ineq:est_norml2chi1}) with (\ref{ineq:est_normh1chi2}) side by side and using \eqref{ineq:poincare1}, it yields
	\begin{align*}
		\|\chi\|^4_{4,h}\leq c_1c_2\tau\|\chi\|^4_{4,h}.
	\end{align*}
 That is, one can achieve the uniqueness of a solution when $\tau$ is small enough. This completes the proof of Theorem  \ref{thm:exist_unique_sol_CNFD} for the existence and uniqueness solution. 
\end{proof}
\subsection{Convergence}

The main results of the paper are presented in the following theorem for convergence.
\begin{thm}[Convergence]\label{thm:est_error_L2H01}
	Let $u$ be an exact solution to (\ref{eq:NLS}). Assume $u$ is smooth enough and $\{U^n\}_{n=0}^N$ satisfies (\ref{CNFD}). Then, if $h$ is sufficiently small and $\tau \lesssim h$, the following estimations hold:
	\begin{align}
		\max_{n\in [1,N]}\|u^n-U^n\|_{2,h} &\lesssim \tau^2+h^2,\label{ineq:errorUandu_L2}\\
		\max_{n\in [1,N]}\|\delta^+ (u^n- U^n)\|_{2,h} &\lesssim \tau^2+h^2.\label{ineq:errorUandu_H01}
	\end{align}
\end{thm}

Before proving Theorem \ref{thm:est_error_L2H01}, we introduce the following settings to rewrite the CNFD scheme (\ref{CNFD}) in terms of globally Lipschitz continuous functions and then establish Proposition \ref{pro3_9}. 

Let $M:=\max\{|u(\xx,t)|:(\xx,t)\in\Omega\times(0,T)\}+1$. We define the auxiliary functions $\tilde{\varphi}, \tilde{\psi}_1,\tilde{\psi}_2:\mathbb{C}\times\mathbb{C}\rightarrow \mathbb{C}$ as follows
\begin{align*}
	\tilde{\varphi}(z,w)&:=\left\{\begin{array}{cc}
		\varphi(z,w) & \text{if }\big|\frac{z+w}{2}\big| \leq M,\\
		\frac{M^2}{2}(z+w) & \text{if }\big|\frac{z+w}{2}\big| > M,
	\end{array}
	\right.\\
	\tilde{\psi}_1(z,w)&:=\left\{\begin{array}{cc}
		\psi_1(z,w) & \text{if }|z|,|w| \leq M,\\
		\frac{F_1(M^2)-F_1(|w|^2)}{M^2-|w|^2}\frac{(z+w)}{2} & \text{if }|z|> M,|w| \leq M,\\
		\frac{F_1(|z|^2)-F_1(M^2)}{|z|^2-M^2}\frac{(z+w)}{2} & \text{if }|z|\leq M,|w| > M,\\
		f_1(M^2)\frac{(z+w)}{2} & \text{if }|z|> M,|w| >  M,
	\end{array}
	\right.\\
	\tilde{\psi}_2(z,w)&:=\left\{\begin{array}{cc}
		\psi_2(z,w) & \text{if }|z|,|w| \leq M,\\
		\frac{F_2(M^2)-F_2(|w|^2)}{M^2-|w|^2}\frac{(z+w)}{2} & \text{if }|z|> M,|w| \leq M,\\
		\frac{F_2(|z|^2)-F_2(M^2)}{|z|^2-M^2}\frac{(z+w)}{2} & \text{if }|z|\leq M,|w| > M,\\
		f_2(M^2)\frac{(z+w)}{2} & \text{if }|z|> M,|w| >  M.
	\end{array}
	\right.
\end{align*}
The functions $\tilde{\varphi}$, $\tilde{\psi}_1$, and $\tilde{\psi}_2$ are globally Lipschitz continuous. Let $V^0:=U^0$ and let $V^n\in X_{JK}$, for $(j,k) \in \T_{JK}$, $n=1,\cdots,N$, satisfy 
\begin{align}
	i\delta_t^+V^n_{j,k}+\delta^2 V_{j,k}^{n+1/2} +\lambda\tilde{\psi}_1(V^{n+1}_{j,k},V^n_{j,k}) -\nu\tilde{\psi}_2(V^{n+1}_{j,k},V^n_{j,k}) +i\varepsilon \tilde{\varphi}(V^{n+1}_{j,k},V^n_{j,k})=0.\label{soltuonV_globallylipschtiz}
\end{align}

\begin{pro} \label{pro3_9}
	Let $u$ be an exact solution to (\ref{eq:NLS}). Assume $u$ is smooth enough and $\{V^n\}_{n=1}^N$ satisfies (\ref{soltuonV_globallylipschtiz}). Then, for sufficiently small $\tau$, the following estimation holds:
	\begin{align}\label{ineq:errorVandu}
		\max_{n\in [1,N]}\|u^n-V^n\|_{2,h} \leq c(\tau^2+h^2),
	\end{align}
	where $c$ is a constant independent of $h$ and $\tau$.
\end{pro}
\begin{proof}
	Let $r^n\in X_{JK}$ be the consistency error of the method (\ref{soltuonV_globallylipschtiz}) or (\ref{CNFD}), i.e, for $(j,k)\in\T_{JK}$ with $u^{n+1/2}=\frac{1}{2}(u^{n+1}+u^n)$, 
	\begin{align}
		r^n_{j,k} = i\delta_t^+u^n_{j,k} + \delta^2 u_{j,k}^{n+1/2} + \lambda\tilde{\psi}_1(u^{n+1}_{j,k},u^n_{j,k}) - \nu\tilde{\psi}_2(u^{n+1}_{j,k},u^n_{j,k}) +i\varepsilon\tilde{\varphi}(u^{n+1}_{j,k},u^n_{j,k}).\label{df:rnjk}
	\end{align}
	Let $e^n:=u^n-V^n\in X_{JK}$, $n = 0,1,\cdots,N$. Then we have
	\begin{align}\label{truncationerrorCNFD}
		i\delta_t^+e^n_{j,k} + \delta^2 e_{j,k}^{n+1/2} + \lambda(\tilde{\psi}_1(u^{n+1}_{j,k},u^n_{j,k})-\tilde{\psi}_1(V^{n+1}_{j,k},V^n_{j,k}))  \nonumber\\
		- \nu(\tilde{\psi}_2(u^{n+1}_{j,k},u^n_{j,k})  -\tilde{\psi}_2(V^{n+1}_{j,k},V^n_{j,k})) +i\varepsilon(\tilde{\varphi}(u^{n+1}_{j,k},u^n_{j,k}) -\tilde{\varphi}(V^{n+1}_{j,k},V^n_{j,k})) - r_{j,k}^n  = 0.
	\end{align}
	We take the inner product $\langle\cdot,\cdot\rangle_h$  with $e^{n+1/2}$ and then take the imaginary parts. By applying the Schwarz inequality and using the Lipschitz of $\tilde{\varphi},\tilde{\psi}_1,\tilde{\psi}_2 $, we obtain 
	\begin{align*}
		\|e^{n+1}\|^2_{2,h} - \|e^{n}\|^2_{2,h} \leq C\tau(\|e^{n+1}\|_{2,h}  + \|e^{n}\|_{2,h} + \|r^n\|_{2,h})\|e^{n+1/2}\|_{2,h}.
	\end{align*}
	Rewriting the previous estimation with inequality $\|e^{n+1/2}\|_{2,h}\leq \frac{1}{2}(\|e^{n+1}\|_{2,h} + \|e^{n}\|_{2,h})$, there holds
	\begin{align*}
		\|e^{n+1}\|_{2,h} - \|e^{n}\|_{2,h} \leq \frac{C}{2}\tau(\|e^{n+1}\|_{2,h}  + \|e^{n}\|_{2,h} + \|r^n\|_{2,h}).
	\end{align*}
	From \ref{lm:est_rn} for the estimation of $ \|r^n\|_{2,h}$, we obtain
	\begin{align*}
		(1-\frac{C}{2}\tau)\|e^{n+1}\|_{2,h}\leq (1+\frac{C}{2}\tau)\|e^{n}\|_{2,h} + \frac{C_1C}{2}\tau(\tau^2+h^2).
	\end{align*}
	The final result follows in view of discrete Gr\"onwall's theorem.
\end{proof}
Now we will prove Theorem \ref{thm:est_error_L2H01}.\\
\begin{proof}[\bf{Proof of Theorem \ref{thm:est_error_L2H01}}]
	For $\omega\in X_{JK}$, applying the following inequality
	\begin{align*}
		\|\omega\|_{\infty,h} \leq \frac{1}{h}\|\omega\|_{2,h},
	\end{align*}
	and from estimation (\ref{ineq:errorVandu}) in Proposition \ref{pro3_9}, it yields
	\begin{align*}
		\|u^n- V^n\|_{\infty,h} \leq c(\frac{\tau^2}{h}+h),~~~\forall~n=0,1,\cdots,N.
	\end{align*}
	This implies
	\begin{align*}
		\|V^n\|_{\infty,h} \leq \|u^n\|_{\infty,h}+c(\frac{\tau^2}{h}+h), ~~~\forall~n=0,1,\cdots,N.
	\end{align*}
	That is, for a   sufficiently small value of $h$ and $\tau\lesssim h$ such that $c(\frac{\tau^2}{h}+h)<1$, it arrives at
	\begin{align*}
		\|V^n\|_{\infty,h}\leq M, ~~~\forall~n =0,1,\cdots,N.
	\end{align*}
	Therefore, $\{V^n\}_{n=0}^{N}$ satisfies (\ref{CNFD}). From the unique solution in Theorem \ref{thm:exist_unique_sol_CNFD}, we have $V^n=U^n$. From this result and from the inequality (\ref{ineq:errorVandu}), the proof of \eqref{ineq:errorUandu_L2} is completed.\\
	Now, we will estimate the error on discrete norm $\|\delta^+\cdot\|_{2,h}$ in \eqref{ineq:errorUandu_H01}. One can write  (\ref{truncationerrorCNFD}) as  follows:
	\begin{align}\label{truncationerrorCNFD1}
		i\delta_t^+e^n_{j,k} + \chi_{j,k}^n + \eta_{j,k}^n -\zeta^n_{j,k}  + \xi_{j,k}^n - r_{j,k}^n  = 0,
	\end{align}
	where $\chi^n, \eta^n, \zeta^n, \xi^n \in X_{JK}$ and 
	\begin{align*}
		\chi^n_{j,k} &= \delta^2 e_{j,k}^{n+1/2},\\
		\eta_{j,k}^n &=\lambda({\psi}_1(u^{n+1}_{j,k},u^n_{j,k}) - {\psi}_1(U^{n+1}_{j,k},U^n_{j,k}) ),\\
		\zeta_{j,k}^n & = \nu ({\psi}_2(u^{n+1}_{j,k},u^n_{j,k}) - {\psi}_2(U^{n+1}_{j,k},U^n_{j,k}) ),\\
		\xi_{j,k}^n & = i\varepsilon({\varphi}(u^{n+1}_{j,k},u^n_{j,k}) -{\varphi}(U^{n+1}_{j,k},U^n_{j,k})).
	\end{align*}
	The terms $\eta_{j,k}^n,\; \zeta_{j,k}^n$  and $\xi^n_{j,k}$ can be written as the following expressions:
	\begin{align*}
		\eta_{j,k}^n &=\frac{\lambda}{2}\left[e^n_{j,k}\overline{u^n_{j,k}}+U^n_{j,k}\overline{e_j^n}+e^{n+1}_{j,k}\overline{u^{n+1}_{j,k}}+U^{n+1}_{j,k}\overline{e^{n+1}_{j,k}}\right]U^{n+1/2}_{j,k} + \frac{\lambda}{2}(|u_{j,k}^n|^2 + |u^{n+1}_{j,k}|^2)e^{n+1/2}_{j,k},
	\end{align*}
	\begin{align*}
		\zeta_{j,k}^n &=\frac{\nu}{3}(|u^{n+1}_{j,k}|^2-|U^{n+1}_{j,k}|^2)(|u^{n+1}_{j,k}|^2+|u^n_{j,k}|^2+|U^{n+1}_{j,k}|^2)U^{n+1/2}_{j,k}\\&+ \frac{\nu}{3}(|u^{n}_{j,k}|^2-|U^{n}_{j,k}|^2)(|u^{n}_{j,k}|^2+|U^n_{j,k}|^2+|U^{n+1}_{j,k}|^2)U^{n+1/2}_{j,k}\\
		&+ \frac{\nu}{3}(|u^{n+1}_{j,k}|^4+|u^{n+1}_{j,k}|^2|u^{n}_{j,k}|^2+ |u^{n}_{j,k}|^4)e^{n+1/2}_{j,k}\\
		& =\frac{\nu}{3}(e^{n+1}_{j,k}\overline{u^{n+1}_{j,k}}+U^{n+1}_{j,k}\overline{e^{n+1}_{j,k}})(|u^{n+1}_{j,k}|^2+|u^n_{j,k}|^2+|U^{n+1}_{j,k}|^2)U^{n+1/2}_{j,k}\\
		&+\frac{\nu}{3}(e^n_{j,k}\overline{u^n_{j,k}}+U^n_{j,k}\overline{e_j^n})(|u^{n}_{j,k}|^2+|U^n_{j,k}|^2+|U^{n+1}_{j,k}|^2)U^{n+1/2}_{j,k}\\
		&+\frac{\nu}{3}(|u^{n+1}_{j,k}|^4+|u^{n+1}_{j,k}|^2|u^{n}_{j,k}|^2+ |u^{n}_{j,k}|^4)e^{n+1/2}_{j,k},
	\end{align*}
	\begin{align*}
		\xi^n_{j,k} = i\varepsilon\left[e^{n+1/2}_{j,k}\overline{u^{n+1/2}_{j,k}}+U^{n+1/2}_{j,k}\overline{e_j^{n+1/2}}\right]U^{n+1/2}_{j,k} + i\varepsilon |u^{n+1/2}_{j,k}|^2e^{n+1/2}_{j,k}.
	\end{align*}
	Since $U^{n}, U^{n+1}, u^{n}, u^{n+1}, \|\delta^+ U^{n+1}\|_{2,h}, \|\delta^+U^{n}\|_{2,h}, \|\delta^+ u^{n+1}\|_{2,h}, \|\delta^+ u^{n}\|_{2,h}$ are bounded, there holds
	\begin{align*}
		\|\delta^+\eta^n\|_{2,h} \lesssim \|\delta^+ e^{n+1}\|_{2,h} + \|\delta^+ e^{n}\|_{2,h},\\
		\|\delta^+\zeta^n\|_{2,h} \lesssim \|\delta^+ e^{n+1}\|_{2,h} + \|\delta^+ e^{n}\|_{2,h},\\
		\|\delta^+\xi^n\|_{2,h} \lesssim \|\delta^+ e^{n+1}\|_{2,h} + \|\delta^+ e^{n}\|_{2,h}.
	\end{align*} 
	Taking in (\ref{truncationerrorCNFD1}) the inner product $\langle\cdot,\cdot\rangle_h$  with $e^{n+1}-e^{n}$ and taking the real part, one can get
	\begin{align}
		\|\delta^+ e^{n+1}\|^2_{2,h}-\|\delta^+ e^{n}\|^2_{2,h} &= 2\Re\langle\eta^n-\zeta^n+\xi^n - r^n,e^{n+1}-e^n\rangle_h\nonumber\\
		&=2\Re\langle \eta^n-\zeta^n+\xi^n - r^n, i\tau(\chi^n+\eta^n+\zeta^n+\xi^n - r^n) \rangle_h\nonumber\\
		&= -2\tau\Im\langle \eta^n-\zeta^n + \xi^n - r^n, \chi^n\rangle_h \label{equ:|e^{n+1}|2_{1,h}-|e^{n}|2_{1,h}}.
	\end{align}
	Using \eqref{discretegreen_x} and \eqref{discretegreen_y}, it yields
	\begin{align*}
		|\langle \eta^n-\zeta^n+\xi^n-r^n, \chi^n\rangle_h| &= |\langle \eta^n-\zeta^n+\xi^n-r^n,\delta^2 e^{n+1/2}\rangle_h|\\
		&= |(-\delta^+\eta^n +\delta^+\zeta^n-\delta^+\xi^n +\delta^+r^n,\delta^+e^{n+1/2})_h|\\
		& \leq  \frac{1}{2}(\|\delta^+ \eta^n\|_{2,h}+\|\delta^+\zeta^n\|_{2,h}+\|\delta^+ \xi^n\|_{2,h}+\|\delta^+ r^n\|_{2,h})\\
		&\times(\|\delta^+ e^{n+1}\|_{2,h}+\|\delta^+ e^{n}\|_{2,h}).
	\end{align*}
	Using the estimations of $\|\delta^+ \eta^n\|_{2,h}$, $\|\delta^+ \zeta^n\|_{2,h}$ $\|\delta^+ \xi^n\|_{2,h}$ and of $\|\delta^+ r^n\|_{2,h}$ in the second estimation of \ref{lm:est_rn}, it yields
	\begin{align}
		|\langle \eta^n+\zeta^n+\xi^n - r^n, \chi^n\rangle_h| \lesssim \|\delta^+ e^{n}\|^2_{2,h} + \|\delta^+ e^{n+1}\|^2_{2,h} + (\tau^2+h^2)^2.\label{est:etazetaxir}
	\end{align}
	From equation \eqref{equ:|e^{n+1}|2_{1,h}-|e^{n}|2_{1,h}}, and inequality \eqref{est:etazetaxir}, it implies
	\begin{align*}
		\|\delta^+ e^{n+1}\|^2_{2,h}-\|\delta^+e^{n}\|^2_{2,h}  \leq \tau(\tau^2+h^2)^2 + \tau(\|\delta^+e^{n+1}\|^2_{2,h} +  \|\delta^+ e^{n}\|^2_{2,h} ).
	\end{align*}
	Using discrete Gr\"onwall's inequality, we obtain
	\begin{align*}
		\|\delta^+ e^{n}\|^2_{2,h} \lesssim (\tau^2+h^2)^2,
	\end{align*}
	which satisfies \eqref{ineq:errorUandu_H01}.\\
 This completes the proof of Theorem \ref{thm:est_error_L2H01}. 
\end{proof}

We conclude this section by discussing a few open questions related to extending the spatial dimension of the NLS equation, modifying the Crank-Nicolson scheme for temporal discretization, and achieving higher accuracy with a fourth-order spatial difference.

\begin{itemize}
    \item {\bf Remark 3.2} We note that the CNFD presented in the current work could be extended to solve 3D damped CQNLS equations. The challenge in extending the current numerical scheme and convergence analysis to the 3D perturbed CQNLS lies in the increased computational complexity due to the quintic nonlinearity, nonlinear damping, and the potential for more intricate nonlinear interactions in a higher dimension. These require some specific estimations for convergence for solving the $(3+1)$D perturbed CQNLS. However, the fundamental principles of the current approach would remain applicable.
    
     \item {\bf Remark 3.3} Recently, for many parabolic equations, it has been discovered that, a modified Crank-Nicolson approximation would lead to a stronger stability property for many parabolic equations, such as an $H^2$ convergence of a second-order convex-splitting finite difference
     scheme for the three-dimensional Cahn-Hilliard equation, stability
     and convergence of a second order mixed finite element method
     for the Cahn-Hilliard equation \cite{Wang2016, Guo2016}. In fact,
     these studies has revealed that if the Crank-Nicolson
    approximation is replaced by $\frac{3}{4} u^{n+1} + \frac{1}{4} u^{n-1}$, instead
    of the standard formulation of $\frac{1}{2} (u^{n+1} + u^n)$, a stronger
    stability can be proved for the parabolic equations \cite{Wang2016, Guo2016}. For a parabolic PDE such as the Cahn-Hilliard equation, the method  requires some assumptions of periodic boundary conditions and convex splitting. As a result, replacing the standard Crank-Nicolson approximation with $\frac{3}{4} u^{n+1} + \frac{1}{4} u^{n-1}$ can lead to stronger stability properties for parabolic PDEs. While the perturbed $(2+1)$D CQNLS equation is not a parabolic PDE, the principles of the modified scheme may still be applicable but it is not straightforward and needs some specific estimations for convergence. 
    
    \item {\bf Remark 3.4} It is worthy to note that for the {\it unperturbed} cubic NLS equation in two dimensions, a fourth-order compact with a three-point stencil difference scheme in space combined with a Crank-Nicolson finite difference in time has been reported by \cite {Wang2012}. In this scheme, the second derivative at a spatial grid point depends on both the function values and the second derivatives at neighboring points and it is an implicit scheme. Additionally, the energy stable fourth order finite difference scheme also has been reported for solving the Cahn-Hilliard equation \cite{Cheng2019}. These findings encourage and support the solvability for implementing the fourth order long stencil difference spatial approximation combined with the Crank-Nicolson temporal discretization for perturbed (2+1)D NLS-type equations. Due to strong nonlinearities of perturbed (2+1)D CQNLS equations, using a standard fourth order long stencil finite difference with a five-point stencil and Crank-Nicolson temporal discretization for perturbed (2+1)D NLS-type equations is not quite similar and requires specific estimations for proving the unique solvability and convergence. It is an open problem for a class of (2+1)D CQNLS equations. 
    
\end{itemize}

\section{Numerical Experiments}
\label{Num}
In this section, we verify the numerical results of the CNFD scheme \eqref{CNFD} by the numerical experiments of  (\ref{eq:NLS}). To find a discrete solution of the CNFD scheme, the nonlinear system \eqref{CNFD} can be linearized by the fixed point method as follows. For $l\geq 0$, the linearization of \eqref{CNFD} is
\begin{align} \label{linearize}
	\left\{
	\begin{array}{c}
		i\frac{U^{{n,l+1}}-U^n}{\tau}+\bigg[\delta^2+\lambda\frac{F_1(|U^{n,l}|^2)-F_1(|U^{n}|^2)}{|U^{n,l}|^2-|U^{n}|^2} -
		\nu\frac{F_2(|U^{n,l}|^2)-F_2(|U^{n}|^2)}{|U^{n,l}|^2-|U^{n}|^2} +
		i\varepsilon|\frac{U^{n,l}+U^{n}}{2}|^{2}\bigg]\frac{U^{l+1,n}+U^{n}}{2} =0,\\
		U^{n,0}= U^n;~U^{n,l+1}\in X_{JK}.
	\end{array}
	\right.
\end{align}
Note that $\{U^{n,l}\}_{l=0,1,\cdots}$, which is solved by \eqref{linearize}, converges to $U^{n+1}$ as $l \to \infty$. In practice, one can choose $U^{n+1} = U^{n,l+1}$ when
\begin{align*}
	\frac{\|U^{n,l+1}-U^{n,l}\|_{2,h}}{\|U^{n,l+1}\|_{2,h}} \leq 10^{-8}.
\end{align*}
For comparison, the numerical ``exact" solution $u^{(num)}_e$ is obtained by the simulation of  (\ref{eq:NLS}) using the split-step Fourier method (SSFM) with the second-order
accuracy \cite{HN2021,Yang2010}.
Let us define the relative error (R.E.) in calculations of the solution $U(x,y,t)$ as follows:
\begin{align*}
	E_{2,h} = \frac{\big\lVert U - u_e^{(num)} \big\rVert_{2,h}}{\big\lVert u_e^{(num)}\big\rVert_{2,h}},\quad
	E_{1,h} = \frac{\big\|\delta^+( U - u_e^{(num)})\big\|_{2,h}}{\big\|\delta^+ u_e^{(num)} \big\|_{2,h}}.
\end{align*}
The analysis results are validated by the following two numerical tests: Test 1 for an initial condition of a 2D soliton and Test 2 for a nonsolitonic initial condition of a Gaussian beam. The CNFD scheme and all numerical experiments are implemented in Matlab. The CNFD algorithm can be summarized as follows in Algorithm \ref{alg:cap}. 

\begin{algorithm}
\caption{CNFD algorithm}\label{alg:cap}
\begin{algorithmic}

\State {\bf Input:} Spatial domain $\Omega$, time interval $[0, T]$, spatial mesh size $h$, time step size $\tau$, initial condition of $u_0(\xx)$.
\State Let $n=0$ and $U^0=u_0(\xx)$.
\While{$n<N$} \Comment{N=$T/\tau$}
\State Let $l=0$ and $U^{n,0}=U^n$.
\State Calculate $U^{n,1}$ using Eq. (\ref{linearize}).
  \While{
  $\frac{\|U^{n,l+1}-U^{n,l}\|_{2,h}}{\|U^{n,l+1}\|_{2,h}} > 10^{-8} $}
   \State Let $l=l+1$.
    \State Calculate $U^{n,l+1}$ using Eq. (\ref{linearize}).
\EndWhile
 \State $U^{n+1} = U^{n,l+1}$.
\State Let $n = n+1$.
\EndWhile
\State {\bf Output:} $U=U^{N}$.
\end{algorithmic}
\end{algorithm}

\subsection{Test 1}
In this test, we consider the initial soliton condition for the simulations in the form
\begin{align}
	u_0(x,y) =A_0 v_0(X_0,Y_0)\exp\left[i\alpha_0 + i\chi_0({X_0},{Y_0})\right]. 
	\label{IC}
\end{align}
In (\ref{IC}), 
$A_0$ is the initial amplitude parameter of the soliton,
$X_0 = x-x_0$, $Y_0=y-y_0$,
$\chi_0({X_0},{Y_0}) = {d}_{1}X_0/2+{d}_2Y_0/2$,
where $(x_0,y_0)$ is the initial position and  ${\bf d} = (d_1,d_2)$ is the velocity vector of the soliton with the velocity components in the $x$ and $y$ directions as $d_1$ and $d_2$, respectively, 
$\alpha_0$ is the initial phase, and
$v_0(X_0,Y_0)$ is the localized real-valued amplitude function.
Using $A_0=1$ for simplicity and substituting (\ref{IC}) into (\ref{eq:NLS}) with $\varepsilon=0$, it implies the following elliptic equation for $v_0$  
\cite{HN2021,Yang2010}:
\begin{align}
	\Delta v_0 + \lambda v_0^3 - \nu v_0^5= \mu v_0,
	\label{elliptic_eq1}
\end{align}
where $\mu$ is the propagation constant.
To define the ground state $v_{0}$, one can numerically solved 
(\ref{elliptic_eq1}) using a Fourier iteration method such as the accelerated imaginary-time evolution method (AITEM) proposed by \cite{Yang2008, Yang2010}. 
\begin{figure}[!h]
	\begin{center}
		\epsfxsize=14.4cm  \epsffile{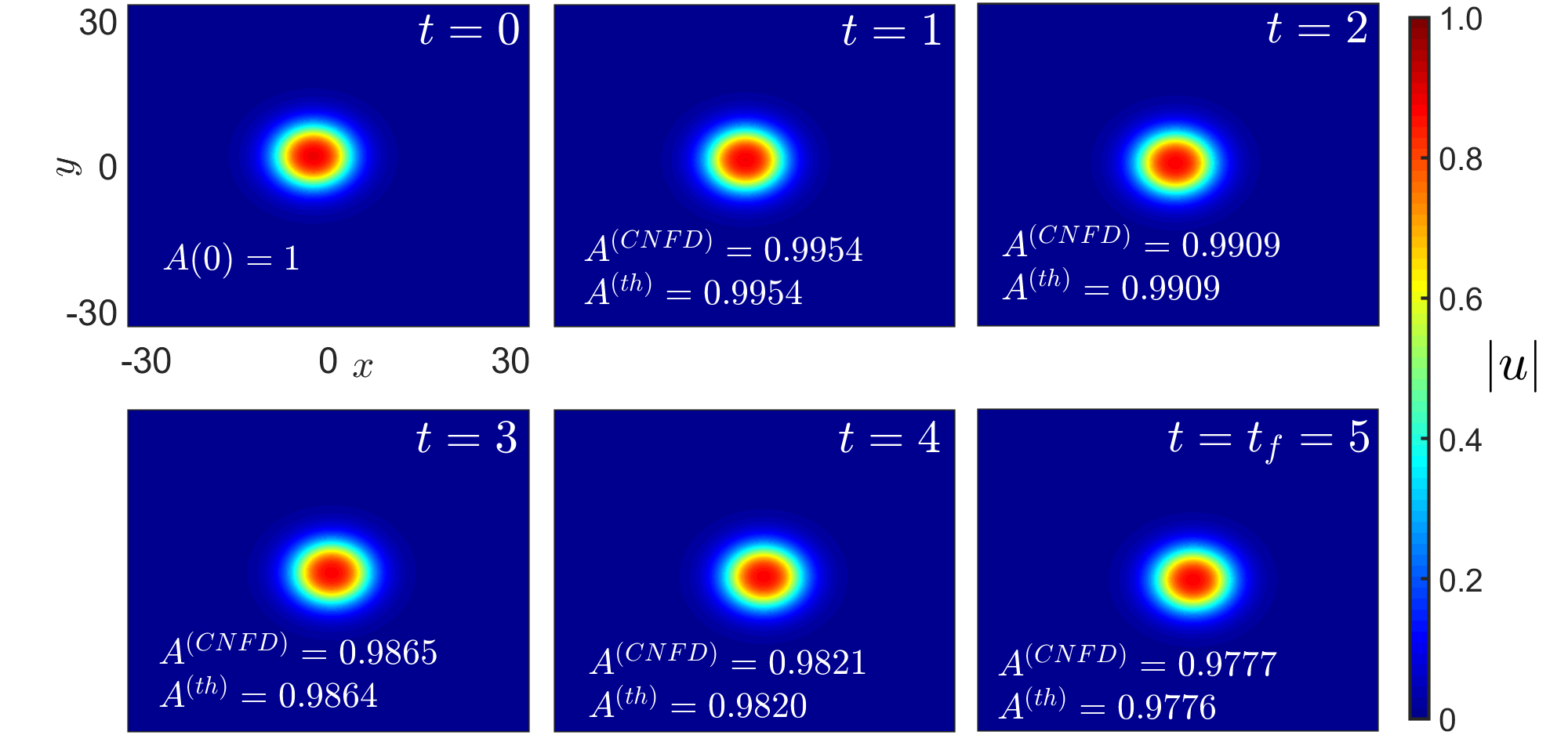} 
	\end{center}
	\caption{(Color online) The contour plot of the initial soliton profile $|u_0(x,y)|$ of (\ref{IC}) and the evolution of its profiles $|u(x,y,t)|$ obtained by simulating (\ref{eq:NLS}) with the CNFD scheme (\ref{CNFD}) for  $h=2^{-4}$ and $\tau=2^{-7}$.
	}
	\label{fig1}
\end{figure}
\begin{table}[!h]
	\centering
	\begin{tabular}{ |c|c|c|c|c|c|} 
		\hline Time
		& R.E. \& & $h=2^{-2}$ &  $h=2^{-3}$ &  $h=2^{-4}$ &  $h=2^{-5}$ \\  
		& Rate  & $\tau=2^{-5}$ &  $\tau=2^{-6}$ &  $\tau=2^{-7}$ &  $\tau=2^{-8}$ \\ 
		\hline
		\multirow{4}{*}{t=1}
		&$E_{2,h}$&1.227E-03  &3.110E-04  &7.992E-05  &2.123E-05 \\ 
		& Rate    & 1.973    &1.946     &1.883     & \\ 
		&$E_{1,h}$&1.851E-03  &4.656E-04  &1.176E-04  &3.022E-05 \\ 
		& Rate    &1.988     &1.980     &1.946     & \\ 
		\hline
		\multirow{4}{*}{t=2}
		&$E_{2,h}$&2.418E-03  &6.122E-04  &1.570E-04  &4.148E-05 \\ 
		& Rate    &1.975     &1.949     &1.893     & \\ 
		&$E_{1,h}$&3.478E-03  &8.751E-04  &2.218E-04  &5.730E-05 \\ 
		& Rate    &1.987     &1.973     &1.935     & \\ 
		\hline
		\multirow{4}{*}{t=3}
		&$E_{2,h}$&3.566E-03  &9.013E-04  &2.305E-04  &6.044E-05 \\ 
		& Rate    &1.978     &1.955     &1.907     & \\ 
		&$E_{1,h}$&4.916E-03  &1.237E-03  &3.141E-04  &8.132E-05 \\ 
		& Rate    &1.987     &1.969     &1.931  & \\ 
		\hline
		\multirow{4}{*}{t=4}
		&$E_{2,h}$&4.658E-03  &1.175E-03  &2.994E-04  &7.793E-05 \\ 
		& Rate    &1.982     &1.962     &1.921     & \\ 
		&$E_{1,h}$&6.212E-03  &1.563E-03  &3.971E-04  &1.028E-04 \\ 
		& Rate    &1.987     &1.969     &1.932     & \\ 
		\hline
		\multirow{4}{*}{t=5}
		&$E_{2,h}$&5.687E-03  &1.432E-03  &3.638E-04  &9.406E-05 \\ 
		& Rate    &1.986     &1.968     &1.934     & \\ 
		&$E_{1,h}$&7.458E-03  &1.877E-03  &4.763E-04  &1.229E-04 \\ 
		& Rate    &1.987     &1.971     &1.938     & \\ 
		\hline
	\end{tabular}
	\caption{The relative errors in measurement of the solution $u(x,y,t)$ and the convergence rate of the CNFD scheme (\ref{CNFD}) with the solitonic initial condition of (\ref{IC}).}\label{table1}
\end{table}

\begin{table}[!h]
	\centering
	\begin{tabular}{ |c|c|c|c|c|} 
		\hline 
		 The schemes & $h=2^{-2}$ &  $h=2^{-3}$ &  $h=2^{-4}$ &  $h=2^{-5}$ \\  
		   & $\tau=2^{-5}$ &  $\tau=2^{-6}$ &  $\tau=2^{-7}$ &  $\tau=2^{-8}$ \\ 
		\hline
		CNFD& 11.604 & 85.374 &759.680  & 7045.008\\ 
		 SSFM    & 0.932    &6.576     & 70.216    & 685.598\\ 
		\hline
	\end{tabular}
	\caption{The computational time (measured in seconds) for simulating (\ref{eq:NLS}) for $t=1$ using the CNFD scheme and using the SSFM scheme.}\label{table1_add}
\end{table}

As a concrete example, the simulations are performed with $\lambda =1$, $\nu =1$, and $\varepsilon = 0.01$ on the computational spatial domain $\Omega = [-60,60]\times [-60,60]$. 
One can use the parameters for the initial condition as $(x_0,y_0)=(-2.5, 2.0)$, $(d_1,d_2)=(1, -0.8)$, and $\alpha_0=0$. 
For solving (\ref{elliptic_eq1}) to compute  the ground state $v_0$ in (\ref{IC}), the input function $\tilde v_0 = \sech(X_0^2 + Y_0^2)$
and the beam power value $\tilde P_0 =60$ are numerically generated by using the AITEM scheme. As a result, it yields the power value of the initial soliton $u_0$ as $P=60$ at $\mu=0.1415$.  

First, the simulation results are presented for the mesh size $h=2^{-4}$ and the time step $\tau=2^{-7}$. 
We compare the amplitude parameter $A(t)$ of solitons measured by the CNFD scheme to the ones calculated by the theoretical calculation and by the SSFM scheme, where the theoretical expression for $A(t)$ is given by  \cite{HN2021} 
\begin{align}
	A(t) = A_0\left[1 + 2\epsilon \|v_0\|^{4}_{4}\|v_0\|^{-2}_{2} A_0^4 t \right]^{-1/2}.
	\label{Ath}
\end{align}
The initial soliton profile $|u_0(x,y)|$ of (\ref{IC}) and the evolution of its profiles $|u(x,y,t)|$ measured by the simulation of (\ref{eq:NLS}) using the CNFD scheme of (\ref{CNFD}) are presented in Fig. \ref{fig1}. At each time $t=0,1,2,3,4,$ and at the final time $t_f=5$, the contour plot of the soliton profile is depicted. 
Furthermore, the soliton amplitude parameters $A^{(CNFD)}$ and $A^{(th)}$ are calculated, where $A^{(CNFD)}$ is measured by using the CNFD scheme and $A^{(th)}(t)$ is calculated from (\ref{Ath}).
As can be seen in Fig. \ref{fig1}, the agreement between
the values of $A^{(CNFD)}(t)$ and $A^{(th)}(t)$ is very good. In fact, the relative error in measurement of the soliton amplitude $A(t)$ from the CNFD scheme and the theoretical calculation of (\ref{Ath}), which is defined by $|A^{(CNFD)}(t) - A^{(th)}(t)|/A^{(th)}(t)$, is less than $1.1642$E-4 for $0<t\le t_f$.  
Additionally, we also observe that the relative error in measuring $A^{(CNFD)}(t)$ 
and $A_e(t)$, which is defined by $|A^{(CNFD)}(t) - A_e(t)|/A_e(t)$, is less than $4.1338$E-7 for $0<t\le t_f$, where $A_e$ is calculated by using the SSFM scheme, i.e., $A_e$ is the amplitude parameter of $u^{(num)}_e$.

Second, the simulation results  for varying values of the mesh size $h$ and the time step $\tau$ are presented.
The values of $E_{2,h}$ and $E_{1,h}$ are presented in Table \ref{table1}.  Moreover, on the lines 2 and 4 of each row of the table, one can observe the second order convergence rate of the CNFD scheme for $\|\cdot\|_{2,h}$-norm and $\|\delta^+\cdot\|_{2,h}$-norm, respectively. 
From the observations in simulation results above, our CNFD scheme of (\ref{CNFD}) was validated.
Additionally, the computational time required to simulate equation (\ref{eq:NLS}) for $t=1$ by using the CNFD scheme (\ref{CNFD}) and the SSFM scheme is also presented in Table  \ref{table1_add}. The central processing unit (CPU) used for the simulations is an Intel(R) Core(TM) i5-835U @1.70GHz-1.90GHz. It can be observed that the computational time required using the SSFM scheme is less than that required when using the CNFD scheme. However, as previously mentioned, unlike the CNFD scheme, the SSFM scheme does not hold the conservation of energy \cite{Bao2013a}.

\begin{figure}[!h]
	\begin{center}
		\epsfxsize=12cm  \epsffile{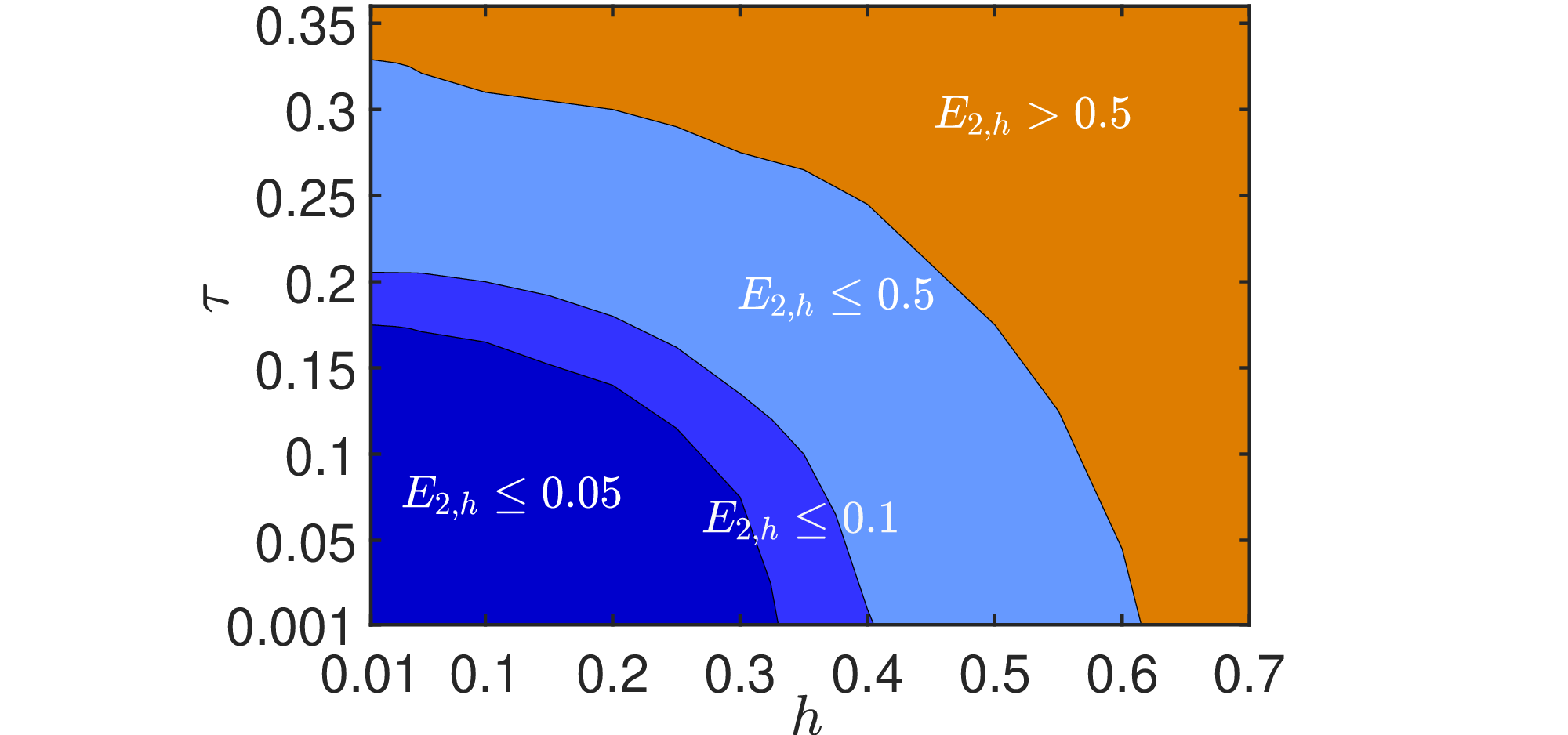} 
	\end{center}
	\caption{(Color online) The relative error $E_{2,h}$ in different regions of $h$ and $\tau$ in simulations using the CNFD scheme with the solitonic initial condition of (\ref{IC}).
	}
	\label{fig1_add}
\end{figure}
\begin{figure}[!h]
	\begin{center}
		\epsfxsize=14.4cm  \epsffile{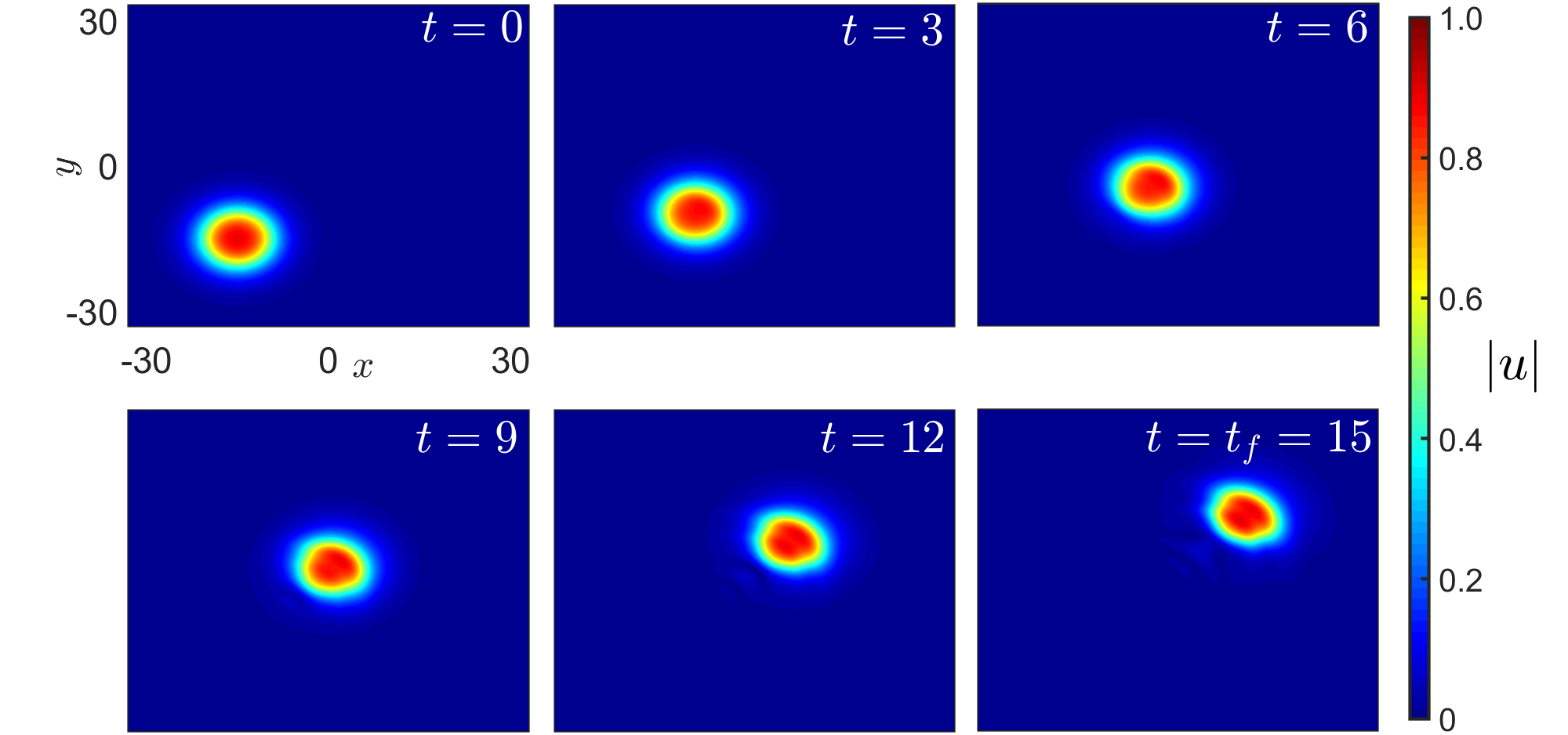} 
	\end{center}
	\caption{(Color online) An illustration for the oscillations in the numerical solution using the CNFD scheme. The contour plots represent the initial soliton profile $|u_0(x,y)|$ of (\ref{IC}) and the evolution of its profiles $|u(x,y,t)|$ obtained by simulating (\ref{eq:NLS}) with $h=2^{-3}$ and $\tau=3h$. 
	}
	\label{fig2_add}
\end{figure}

Finally, we analyze the relative errors of numerical solutions with varying $h$ and $\tau$. 
The parameters for simulations are $(x_0,y_0)=(-15, -15)$, $(d_1,d_2)=(2, 2)$, $\varepsilon=0.01$, $t_f\ge 10$, $0<h\le 0.9$, and $0<\tau \le 0.5$. Other parameters are the same as ones used in Fig. \ref{fig1} and Table \ref{table1}. Figure \ref{fig1_add} shows 
the parameter space regions of $h$ and $\tau$ where $E_{2,h}\le 0.05$ (the dark blue region), $E_{2,h}\le 0.1$ (the ultramarine blue region), $E_{2,h}\le 0.5$ (the bright blue region), and $E_{2,h}> 0.5$ (the orange region) in simulations with $t_f=10$.
One can observe that $E_{2,h}\le 0.5$ for all small enough values of $h $ and $\tau \lesssim h$. 
Additionally,  to illustrate the oscillations in the numerical solution using the CNFD scheme, 
Fig. \ref{fig2_add} shows the contour plot of the initial soliton profile $|u_0(x,y)|$ of (\ref{IC}) and the evolution of its profiles $|u(x,y,t)|$  with $h=2^{-3}$ and $\tau=3h$. As can be seen, the oscillations of the numerical solution can be clearly observed for $t\ge 9$ with $E_{2,h} > 0.5$.

\subsection{Test 2}
In the second test, we consider the following nonsolitonic initial condition of a Gaussian beam:
\begin{align}
	u_0(x,y)= \frac{2}{\sqrt{\pi}}(x+iy)e^{-\frac{1}{2}(x^2+y^2)}.
	\label{IC2}
\end{align}
We take $\Omega = [-10,10]\times[-10,10]$, $\varepsilon=0.01$, and different values for $\lambda$ and $\nu$. The mesh size and the time size are similar as ones in Test 1. Table \ref{table2} and Table \ref{table3} show the relative error in measurement of the solution $u(x,y,t)$ at the final time of the simulation $T=0.5$. One can observe the order of the convergence in both norms $\|\cdot\|_{2,h}$ and $\|\delta^+\cdot\|_{2,h}$ are approximately of order two for varying values of $\lambda$ and $\nu$ and increase to two while the mesh size and time step tend to 0. Thus, the theoretical results on the order of the convergence in Theorem \ref{thm:est_error_L2H01} are verified. Moreover, in Table \ref{table2} the errors decrease with $\nu\leq 1$ and increase with $\nu >1$. And in Table \ref{table3}, the errors increase along the  $|\lambda|$ from $0.01$ to $5$.

\begin{table}[!ht]
	\centering
	\begin{tabular}{ |c|c|c|c|c|c|} 
		\hline $\lambda,\nu$
		& R.E. \& & $h=2^{-2}$ &  $h=2^{-3}$ &  $h=2^{-4}$ &  $h=2^{-5}$ \\  
		& Rate  & $\tau=2^{-5}$ &  $\tau=2^{-6}$ &  $\tau=2^{-7}$ &  $\tau=2^{-8}$ \\ 
		\hline
		\multirow{4}{3.5em}{$\lambda=1$, $\nu=0.01$}
		&$E_{2,h}$&4.840E-02  &1.236E-02  &3.104E-03  &7.757E-04 \\ 
		& Rate    &1.957     &1.991  &2.001  & \\ 
		&$E_{1,h}$&4.444E-02  &1.183E-02  &3.004E-03  &7.528E-04 \\ 
		& Rate    &1.878     &1.969     &1.995 & \\ 
		\hline
		\multirow{4}{3.5em}{$\lambda=1$, $\nu=0.05$}
		&$E_{2,h}$&4.835E-02  &1.234E-02  &2.989E-03  &7.490E-04 \\ 
		& Rate    &1.958     &1.991  &2.001  & \\ 
		&$E_{1,h}$&4.430E-02  &1.177E-02  &3.100E-03  &7.746E-04 \\ 
		& Rate    &1.881     &1.970     &1.995 & \\ 
		\hline
		\multirow{4}{3em}{$\lambda=1$, $\nu=0.1$}
		&$E_{2,h}$&4.829E-02  &1.232E-02  &3.094E-03  &7.734E-04 \\ 
		& Rate    &1.959     &1.991     &2.001  & \\ 
		&$E_{1,h}$&4.412E-02  &1.170E-02  &2.971E-03  &7.445E-04 \\ 
		& Rate    &1.885     &1.970     &1.995  & \\ 
		\hline
		\multirow{4}{3em}{$\lambda=1$, $\nu=0.5$}
		&$E_{2,h}$&4.790E-02  &1.219E-02  &3.062E-03  &7.657E-04 \\ 
		& Rate    &1.965     &1.991     &1.999  & \\ 
		&$E_{1,h}$&4.291E-02  &1.123E-02  &2.861E-03  &7.191E-04 \\ 
		& Rate    &1.911     &1.961     &1.990  & \\ 
		\hline
		\multirow{4}{3em}{$\lambda=1$, $\nu=1$}
		&$E_{2,h}$&4.766E-02  &1.211E-02  &3.046E-03  &7.627E-04 \\ 
		& Rate    &1.968  &1.987  &1.997  & \\ 
		&$E_{1,h}$&4.201E-02  &1.089E-02  &2.824E-03  &7.156E-04 \\ 
		& Rate    &1.928     &1.929     &1.973  & \\ 
		\hline
		\multirow{4}{3em}{$\lambda=1$, $\nu=2.5$}
		&$E_{2,h}$&4.846E-02  &1.242E-02  &3.167E-03  &7.985E-04 \\ 
		& Rate    &1.951  &1.961  &1.983  & \\ 
		&$E_{1,h}$&4.354E-02  &1.180E-02  &3.372E-03  &8.886E-04 \\ 
		& Rate    &1.845     &1.749     &1.897  & \\ 
		\hline
		\multirow{4}{3em}{$\lambda=1$, $\nu=5$}
		&$E_{2,h}$&5.450E-02  &1.469E-02  &3.882E-03  &9.899E-04 \\ 
		& Rate    &1.855  &1.892  &1.961  & \\ 
		&$E_{1,h}$&5.749E-02  &1.794E-02  &5.627E-03  &1.514E-03 \\ 
		& Rate    &1.602     &1.594     &1.858  & \\ 
		\hline
	\end{tabular}
	\caption{
		The relative errors in measurement of the solution $u(x,y,t)$ and the convergence rate of the CNFD scheme (\ref{CNFD}) with the nonsolitonic initial condition of (\ref{IC2}).
	}\label{table2}
\end{table}

\begin{table}[th]
	\centering
	\begin{tabular}{ |c|c|c|c|c|c|} 
		\hline $\lambda,\nu$
		& R.E. \& & $h=2^{-2}$ &  $h=2^{-3}$ &  $h=2^{-4}$ &  $h=2^{-5}$ \\  
		& Rate  & $\tau=2^{-5}$ &  $\tau=2^{-6}$ &  $\tau=2^{-7}$ &  $\tau=2^{-8}$ \\ 
		\hline
		\multirow{4}{3.9em}{$\lambda=0.01$, $\nu=0.1$}
		&$E_{2,h}$&4.822E-02  &1.224E-02  &3.073E-03  &7.690E-04 \\ 
		& Rate    &1.969     &1.992     &1.998  & \\ 
		&$E_{1,h}$&4.189E-02  &1.075E-02  &2.707E-03  &6.780E-04 \\ 
		& Rate    &1.948     &1.986     &1.996  & \\  
		\hline
		\multirow{4}{3.5em}{$\lambda=0.1$, $\nu=0.1$}
		&$E_{2,h}$&4.813E-02  &1.222E-02  &3.066E-03  &7.673E-04 \\ 
		& Rate    &1.970     &1.992     &1.998  & \\ 
		&$E_{1,h}$&4.186E-02  &1.075E-02  &2.705E-03  &6.773E-04 \\ 
		& Rate    &1.948     &1.986     &1.997  & \\  
		\hline
		\multirow{4}{3.5em}{$\lambda=0.5$, $\nu=0.1$}
		&$E_{2,h}$&4.796E-02  &1.219E-02  &3.058E-03  &7.648E-04 \\ 
		& Rate    &1.968     &1.993     &1.999  & \\ 
		&$E_{1,h}$&4.231E-02  &1.095E-02  &2.761E-03  &6.915E-04 \\ 
		& Rate    &1.932     &1.982     &1.997  & \\  
		\hline
		\multirow{4}{3em}{$\lambda=1$, $\nu=0.1$}
		&$E_{2,h}$&4.829E-02  &1.232E-02  &3.094E-03  &7.734E-04 \\ 
		& Rate    &1.959     &1.991     &2.001  & \\ 
		&$E_{1,h}$&4.412E-02  &1.170E-02  &2.971E-03  &7.445E-04 \\ 
		& Rate    &1.885     &1.970     &1.995  & \\ 
		\hline
        \multirow{4}{3.5em}{$\lambda=-1$, $\nu=0.1$}
		&$E_{2,h}$&5.051E-02  &1.296E-02  &3.263E-03  &8.191E-04 \\ 
		& Rate    &1.948  &1.986  &1.992  & \\ 
		&$E_{1,h}$&4.530E-02  &1.217E-02  &3.098E-03  &7.789E-04 \\ 
		& Rate    &1.861     &1.964     &1.989  & \\ 
		\hline
		\multirow{4}{3em}{$\lambda=5$, $\nu=0.1$}
		&$E_{2,h}$&6.750E-02  &1.816E-02  &4.604E-03  &1.148E-03 \\ 
		& Rate    &1.859  &1.972  &2.005  & \\ 
		&$E_{1,h}$&8.450E-02  &2.584E-02  &6.794E-03  &1.713E-03 \\ 
		& Rate    &1.635     &1.902     &1.983  & \\ 
		\hline
        \multirow{4}{3.5em}{$\lambda=-5$, $\nu=0.1$}
		&$E_{2,h}$&7.816E-02  &2.231E-02 & 5.740E-03  &1.450E-03 \\ 
		& Rate    &1.752  &1.943  &1.980  & \\ 
		&$E_{1,h}$&9.070E-02  &2.982E-02  &7.899E-03  &2.000E-03 \\ 
		& Rate    &1.521     &1.888     &1.974  & \\ 
		\hline
	\end{tabular}
	\caption{
		The relative errors in measurement of the solution $u(x,y,t)$ and the convergence rate of the CNFD scheme (\ref{CNFD}) with the nonsolitonic initial condition of (\ref{IC2}).
	}\label{table3}
\end{table}

We next consider the discrete energy and mass which are defined in \eqref{df:E_h(U^n)} and right hand side of \eqref{discreteConservativeMass}, respectively. In Fig. \ref{fig:discrete_energy} with $(\lambda,\nu,\varepsilon) = (10,1,5)$, the discrete mass decreases along the time and the discrete energy increases at the beginning time and decreases at the end time. Conversely, the discrete energy and discrete mass are constants along the time with parameter $(\lambda, \nu,\varepsilon)= (10,1,0)$.
\begin{figure}[!ht]
	\centering
	\includegraphics[width=0.9\linewidth]{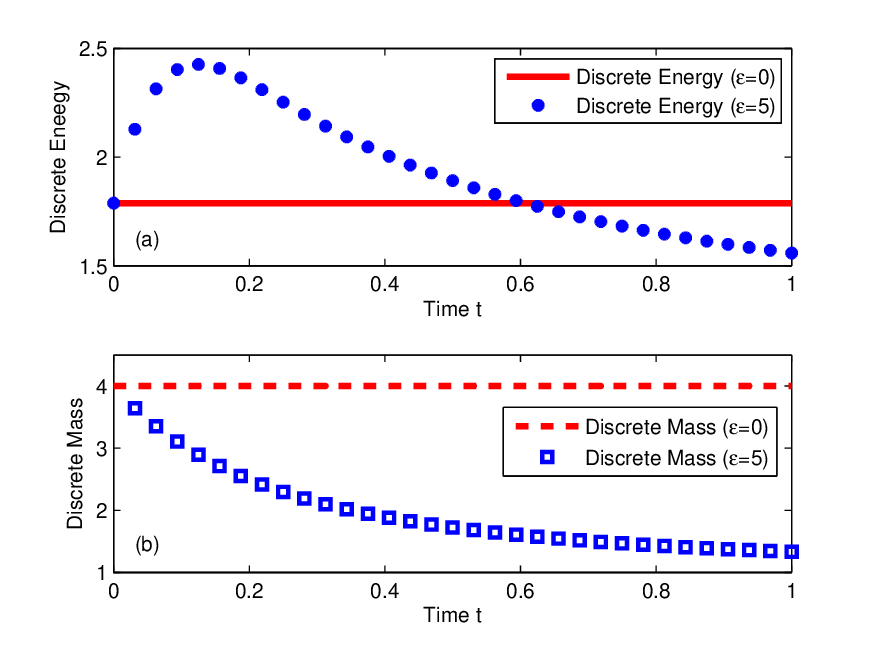}
	\caption{The dependence of discrete energy (a) and discrete mass (b) with respect on time $t$ with $(\lambda,\nu,\varepsilon)= (10,1,5)$ and $(\lambda,\nu,\varepsilon)= (10,1,0)$}
	\label{fig:discrete_energy}
\end{figure}

Finally, we validate the discrete consevation of mass and of energy by calculating the errors between the discrete energy and  continuous energy, and of the discrete mass and continuous mass when $\varepsilon = 0$. The errors between the discrete mass and  energy with continuous mass are defined respectively as following 
\begin{align}
	E^n_{M,h} &= \left|\|U^n\|^2_{2,h}-\|u_0\|^2_{L^2(\Omega)}\right| \label{df_error_mass},\\
	E_{E,h}^n &= \left|E_h(U^n)-E(u_0)\right|, \label{df_error_enegy}
\end{align}
where
\begin{align*}
	E(u_0) = \|\nabla u_0\|^2_{L^2(\Omega}-\lambda\|F_1(|u_0|^2)\|_{L^1(\Omega)}+\nu \|F_2(|u_0|^2)\|_{L^1(\Omega},
\end{align*}
with $F_1$ and $F_2$ are defined in \eqref{df_f1f2_F1F2}. In Table \ref{table7}, the errors between the discrete and continuous mass are extremely small of order $\sim 10^{-13}$ along the time $[0,1]$. This indicates that the errors are essentially impacted by round-off error and the error for solving nonlinear systems. The results demonstrate that the errors between the discrete and continuous energy are almost constant along the time $[0,1]$ and converge in the second order. 
\begin{table}
	\centering
	\begin{tabular}{ |c|c|c|c|c|c|} 
		\hline Time
		& R.E. \& & $h=2^{-2}$ &  $h=2^{-3}$ &  $h=2^{-4}$ &  $h=2^{-5}$ \\  
		& Rate  & $\tau=2^{-5}$ &  $\tau=2^{-6}$ &  $\tau=2^{-7}$ &  $\tau=2^{-8}$ \\ 
		\hline
		\multirow{4}{*}{t=0.25}
		&$E_{M,h}$ &7.37E-14  &6.22E-15  &1.49E-13  &3.20E-14 \\ 
		&$E_{E,h}$ &9.32E-02  &2.34E-02  &5.86E-03 &1.46E-03 \\ 
		& Rate     &1.991  &1.997  &1.999  & \\ 
		\hline
		\multirow{4}{*}{t=0.5}
		&$E_{M,h}$ &7.55E-14  &3.82E-14 & 5.52E-13  &5.60E-13 \\ 
		&$E_{E,h}$ &9.32E-02  &2.34E-02  &5.86E-03 & 1.46E-03 \\ 
		& Rate     &1.991  &1.997  &1.999  & \\ 
		\hline
		\multirow{4}{*}{t=0.75}
		&$E_{M,h}$ &9.50E-14  &2.98E-13  &7.25E-13  &6.71E-13 \\ 
		&$E_{E,h}$ &9.32E-02  &2.34E-02  &5.86E-03 & 1.46E-03 \\ 
		& Rate     &1.991  &1.997  &1.999  & \\ 
		\hline
		\multirow{4}{*}{t=1}
		&$E_{M,h}$ &1.71E-13  &5.84E-13  &9.05E-13 &1.49E-12 \\ 
		&$E_{E,h}$ &9.32E-02  &2.34E-02  &1.47E-03 & 1.467E-03 \\ 
		& Rate     &1.990  &1.995  &1.996  & \\
		\hline
	\end{tabular}
	\caption{The errors between the discrete mass and energy with their continuous parts are defined in \eqref{df_error_mass} and \eqref{df_error_enegy}. }\label{table7}
\end{table}

\section{Conclusion}
\label{Con}

The (2+1)D cubic-quintic NLS equation can stabilize the 2D solitons against its collapse at high speed and high power. In this work, we proposed and analyzed the CNFD scheme for the (2+1)D CQNLS equation with cubic damping, which is of a common class of (2+1)D NLS equation with modified nonlinearity. We showed the existence of a solution by using the Brouwer fixed point theorem. We established a few appropriate settings and estimations to prove the uniqueness of the numerical solution of the (2+1)D CQNLS equation. Moreover, under a mild condition of $\tau\lesssim h$, we showed that the convergence rate is of order $O(\tau^2+h^2)$ in the $L^2-$norm and discrete $H^1-$norm. To validate the numerical scheme and the convergence order, we intensively simulated the (2+1)D CQNLS model by the CNFD scheme for an initial condition of a 2D soliton and also for a nonsolitonic initial condition of a Gaussian beam with varying spatial mesh size $h$ and time step $\tau$. Additionally, we implemented the CNFD to validate a theoretical expression for the amplitude dynamics of single 2D soliton propagation under the effect of cubic damping. Furthermore, we also verified the numerical results of the CNFD scheme with the conventional SSFM scheme. We obtained an excellent agreement between the numerical simulation results and the theoretical results. We also validated that the discrete energy and discrete mass are constants along the time for the unperturbed CQNLS with $\varepsilon=0$. The current work has reported for the first time the rigorous analysis of the CNFD scheme for a perturbed $(2+1)$D CQNLS equation with nonlinear damping. This work significantly contributed to the numerical approaches for a class of $(n+1)$D NLS equations with modified nonlinearity in a high dimensional space. We expect that it also opens a way to solve the $(3+1)$D perturbed CQNLS and the possibility to extend the Crank-Nicolson scheme in time-space by a modified Crank-Nicolson scheme of $\delta u^{n+1} +(1-\delta) u^{n-1}$, where $0<\delta<1$, for the perturbed $(2+1)$D CQNLS equation and for related NLS-type equation in a higher spatial dimension.

\appendix 
\section{Estimate the consistency error}\label{Appendix}
In this Appendix, we establish the estimations of the consistency error of the scheme (\ref{soltuonV_globallylipschtiz}) or (\ref{CNFD}), which is used to prove the convergence rate in the $L^2-$norm and discrete $H^1-$norm in Theorem \ref{thm:est_error_L2H01}.
\begin{App}\label{lm:est_rn}
	Let $u$ be an exact solution to (\ref{eq:NLS}) and let $r^n\in X_{JK}$ defined by \eqref{df:rnjk} be the consistency error of the scheme of (\ref{CNFD}) or (\ref{soltuonV_globallylipschtiz}). If $u$ is smooth enough, then there holds
	\begin{align*}
		&\max_{j,k,n}|r^n_{j,k}| \lesssim \tau^2+h^2,\\
		&\max_{j,k,n}|\delta^+r^n_{j,k}| \lesssim \tau^2+h^2.
	\end{align*}
\end{App}
\begin{proof}
	From the denotation of $r_{j,k}^n$, for all $(j,k)\in\T_{JK}$ and $n\in\overline{0,N-1}$, it yields
	\begin{align}
		\!r^n_{j,k} = i\delta^+_t u^{n}_{j,k}
		+\frac{1}{2}(\delta^2_{x}u^{n+1}_{j,k} + \delta^2_{x}u^{n}_{j,k}) + \frac{1}{2}(\delta^2_{y}u^{n+1}_{j,k} + \delta^2_{y}u^{n}_{j,k}) +\lambda\psi_1(u^{n+1}_{j,k},u^{n}_{j,k}) \nonumber\\
		-\nu\psi_2(u^{n+1}_{j,k},u^{n}_{j,k}) +i\varepsilon \varphi(u^{n+1}_{j,k},u^{n}_{j,k}),\label{def_rjkn1}
	\end{align}
	with $u^{n}_{j,k}=u(x_j,y_k,t_{n})$ for all $(j,k)\in\T_{JK}^0$ and $n\in\overline{0,N}$. Now  the terms on the right hand side in \eqref{def_rjkn1} are estimated as follows. Applying the Taylor expansion, there holds
	\begin{align}
		\delta^+_t u^{n}_{j,k} - \partial_t u(x_j,y_k,t_{n+1/2}) = \frac{\tau^2}{8}\int_0^1s^2\int_0^1\partial_{ttt}u(x_j,y_k,(t_{n+1/2}-\frac{s\tau}{2})+s_1s\tau)ds_1ds, \label{ineq:est_dt}
	\end{align}
	where $t_{n+1/2}$ is the midpoint of $(t_{n},t_{n+1})$. It is similar, one has
	\begin{align*}
		&\delta^2_{x}u^{n}_{j,k} - \partial_{xx}u(x_j,y_k,t_{n}) = \frac{h^2}{6}\int_0^1s^3(\partial_{xxxx}u(x_j+sh,y_k,t_n)+ (\partial_{xxxx}u(x_j-sh,y_k,t_n))ds,\\
		&\delta^2_{x}u^{n+1}_{j,k} - \partial_{xx}u(x_j,y_k,t_{n+1}) = \frac{h^2}{6}\int_0^1s^3(\partial_{xxxx}u(x_j+sh,y_k,t_{n+1})+ \partial_{xxxx}u(x_j-sh,y_k,t_{n+1}))ds,\\
		& u_{xx}(x_j,y_k,t_{n}) + \partial_{xx}u(x_j,y_k,t_{n+1})-2\partial_{xx}u(x_j,y_k,t_{n+1/2}) \\
		&= \frac{\tau^2}{2}\int_0^1 s\int_0^1 \partial_{ttxx}u(x_j,y_k,t_{n+1/2}-\frac{s\tau}{2}+ss_1\tau)ds_1ds.
	\end{align*}
	Summing the previous three equations, one can arrive at
	\begin{align}
		\delta^2_{x}u^{n}_{j,k} + &\delta^2_{x}u^{n+1}_{j,k} -2\partial_{xx}u(x_j,y_k,t_{n+1/2})= \frac{h^2}{6}\int_0^1s^3(\partial_{xxxx}u(x_j+sh,y_k,t_n)+ \partial_{xxxx}u(x_j-sh,y_k,t_n))ds\nonumber\\
		&+ \frac{h^2}{6}\int_0^1s^3(\partial_{xxxx}u(x_j+sh,y_k,t_{n+1})+ \partial_{xxxx}u(x_j-sh,y_k,t_{n+1}))ds\nonumber\\
		&+ \frac{\tau^2}{2}\int_0^1 s\int_0^1 \partial_{ttxx}u(x_j,y_k,t_{n+1/2}-\frac{s\tau}{2}+ss_1\tau)ds_1ds.\label{ineq:est_dxx}
	\end{align}
	Similarly, one obtains
	\begin{align}
		\delta^2_{y}u^{n}_{j,k} + &\delta^2_{y}u^{n+1}_{j,k} -2\partial_{yy}u(x_j,y_k,t_{n+1/2})= \frac{h^2}{6}\int_0^1s^3(\partial_{yyyy}u(x_j,y_k+sh,t_n)+ \partial_{yyyy}u(x_j,y_k-sh,t_n))ds\nonumber\\
		&+ \frac{h^2}{6}\int_0^1s^3(\partial_{yyyy}u(x_j,y_k+sh,t_{n+1})+ \partial_{yyyy}u(x_j,y_k-sh,t_{n+1}))ds\nonumber\\
		&+ \frac{\tau^2}{2}\int_0^1 s\int_0^1 \partial_{ttyy}u(x_j,y_k,t_{n+1/2}-\frac{s\tau}{2}+ss_1\tau)ds_1ds.\label{ineq:est_dyy}
	\end{align}
	For $f\in C^3(\mathbb{R},\mathbb{R})$, we have
	\begin{align}
		&\int^1_0f(|u^{n}_{j,k}|^2+s(|u^{n+1}_{j,k}|^2-|u^{n}_{j,k})|^2))ds
		-f\big(\frac{1}{2}(|u^{n+1}_{j,k}|^2+|u^{n}_{j,k}|^2)\big)\nonumber\\
		&= 2(|u(x^{n+1}_{j,k}|^2-|u^{n}_{j,k}|^2)^2\int_0^{1/2}(s-\frac{1}{2})^2\int_0^1s_1\int_0^1 f''(g^n_{j,k}(s,s_1,s_2))ds_2ds_1ds,\label{eq:inter_f}
	\end{align} 
	where
	\begin{align*}
		g^n_{j,k}(s,s_1,s_2) = \frac{1}{2}(|u^{n+1}_{j,k}|^2+|u^{n}_{j,k}|^2)-s_1(s-\frac{1}{2})(|u^{n+1}_{j,k}|^2-|u^{n}_{j,k})|^2)+2s_2s_1(s-\frac{1}{2}(|u^{n+1}_{j,k}|^2-|u^{n}_{j,k}|^2)).
	\end{align*}
	Moreover, it also has
	\begin{align*}
		|u^{n+1}_{j,k}|^2-|u^{n}_{j,k}|^2 =\tau\int_0^1\partial_{t}|u(x_j,y_k,t_n+s\tau)|^2ds.
	\end{align*}
	Therefore,
	\begin{align*}
		&\int^1_0f(|u^{n}_{j,k}|^2+s(|u^{n+1}_{j,k}|^2-|u^{n}_{j,k})|^2))ds
		-f\big(\frac{1}{2}(|u^{n+1}_{j,k}|^2+|u^{n}_{j,k}|^2)\big)\\
		&= 2\tau^2\bigg[\int_0^1\partial_{t}|u(x_j,y_k,t_n+s\tau)|^2ds\bigg]^2\int_0^{1/2}(s-\frac{1}{2})^2\int_0^1s_1\int_0^1 f''(g^n_{j,k}(s,s_1,s_2))ds_2ds_1ds.
	\end{align*} 
	Using the mean value theorem, there holds  
	\begin{align*}
		&f\big(\frac{1}{2}(|u^{n+1}_{j,k}|^2+|u^{n}_{j,k}|^2)\big)-f(|u(x_j,y_k,t_{n+1/2})|^2)
		\\
		&=\big[\frac{1}{2}(|u^{n+1}_{j,k}|^2+|u^{n}_{j,k}|^2)-|u(x_j,y_k,t_{n+1/2})|^2\big]\int_0^1f'(g_{1,j,k}^n(s_3))ds_3,
	\end{align*}
	where 
	\begin{align*}
		g_{1,j,k}^n(s_2) = |u(x_j,y_k,t_{n+1/2})|^2+s_2(\frac{1}{2}(|u^{n+1}_{j,k}|^2+|u^{n}_{j,k}|^2) - |u(x,y,t_{n+1/2})|^2),
	\end{align*}
	and 
	\begin{align*}
		\frac{1}{2}(|u^{n+1}_{j,k}|^2+|u^{n}_{j,k}|^2)-|u(x_j,y_k,t_{n+1/2})|^2=\frac{\tau^2}{2}\int_0^1s\int_0^1\partial_{tt}|u(x_j,y_k,t_{n+1/2}-\frac{s\tau}{2}+s_1s\tau)|^2ds_1ds.
	\end{align*}
	Then
	\begin{align}
		&f\big(\frac{1}{2}(|u^{n+1}_{j,k}|^2+|u^{n}_{j,k}|^2)\big)-f(|u(x_j,y_k,t_{n+1/2})|^2) \nonumber \\
		& = \frac{\tau^2}{2}\int_0^1s\int_0^1\partial_{tt}|u(x_j,y_k,t_{n+1/2}-\frac{sh}{2}+s_1sh)|^2ds_1ds\int_0^1f'(g_{1,j,k}^n(s_2))ds_2.\label{eq:f_midpoint}
	\end{align}
	From \eqref{eq:inter_f}  and \eqref{eq:f_midpoint}, it implies
	\begin{align}
		&\int^1_0f(|u^{n}_{j,k}|^2+s(|u^{n+1}_{j,k}|^2-|u^{n}_{j,k})|^2))ds
		-f(|u(x_j,y_k,t_{n+1/2})|^2)\nonumber\\
		&= 2\tau^2\bigg[\int_0^1\partial_{t}|u(x_j,y_k,t_n+s\tau)|^2ds\bigg]^2\int_0^{1/2}(s-\frac{1}{2})^2\int_0^1s_1\int_0^1 f''(g^n_{j,k}(s,s_1,s_2))ds_2ds_1ds\nonumber\\
		&+\frac{\tau^2}{2}\int_0^1s\int_0^1\partial_{tt}|u(x_j,y_k,t_{n+1/2}-\frac{sh}{2}+s_1sh)|^2ds_1ds\int_0^1f'(g_{1,j,k}^n(s_2))ds_2.\label{eq:inter_f1}
	\end{align}
	Applying the Taylor expansion, it yields
	\begin{align}
		u^{n+1/2}_{j,k} =u(x_j,y_k,t_{n+1/2}) +\frac{\tau^2}{2}\int_0^1s\int_0^1\partial_{tt}u(x_j,y_k,t_n+1/2-\frac{\tau s}{2}+s_1s\tau)ds_1ds.\label{eq:u^{n+1/2}_{j,k}}
	\end{align} 
	Using the equations \eqref{eq:inter_f1} and \eqref{eq:u^{n+1/2}_{j,k}}, one can calculate $\psi_1(u^{n+1}_{j,k},u^{n}_{j,k})$ and  $\psi_2(u^{n+1}_{j,k},u^{n}_{j,k})$ as follows
	\begin{align}
		&u^{n+1/2}_{j,k}\int^1_0f(|u^{n+1}_{j,k}|^2+s(|u^{n+1}_{j,k})|^2-|u^{n}_{j,k}|^2))ds-u(x_j,y_k,t_{n+1/2})f(|u(x_j,y_k,t_{n+1/2})|^2) \nonumber\\
		&=2\tau^2u^{n+1/2}_{j,k}\bigg[\int_0^1\partial_t|u(x_j,y_k,t_n+s\tau)|^2ds\bigg]^2\int_0^{1/2}(s-\frac{1}{2})^2\int_0^1s_1\int_0^1 f''(g^n_{j,k}(s,s_1,s_2))ds_2ds_1ds \nonumber\\
		&+\frac{\tau^2}{2}u^{n+1/2}_{j,k}\int_0^1s\int_0^1\partial_{tt}|u(x_j,y_k,t_{n+1/2}-\frac{sh}{2}+s_1sh)|^2ds_1ds\int_0^1f'(g_{1,j,k}^n(s_2))ds_2 \nonumber \\
		&+\frac{\tau^2}{2}f(|u(x_j,y_k,t_{n+1/2})|^2)\int_0^1 s\int^1_0\partial_{tt}u(x_j,y_k,t_{n+1/2}-\frac{s\tau}{2}+s_1s\tau)ds_1ds .\label{eq:psi}
	\end{align} 
	From  the equality \eqref{eq:u^{n+1/2}_{j,k}}, there holds 
	\begin{align*}
		|u^{n+1/2}_{j,k}|^2& =|u(x_j,y_k,t_{n+1/2})|^2+\tau^2\int_0^1s\int_0^1\Re(\bar{u}(x_j,y_k,t_{n+1/2})\partial_{tt}u(x_j,y_k,t_{n+1/2}-\frac{\tau s}{2}+s_1s\tau))ds_1ds\\
		&+\bigg|\frac{\tau^2}{2}\int_0^1s\int_0^1\partial_{tt}u(x_j,y_k,t_n+1/2-\frac{\tau s}{2}+s_1s\tau)ds_1ds\bigg|^2.
	\end{align*}
	From the previous equation and \eqref{eq:u^{n+1/2}_{j,k}}, the term $\varphi(u^{n+1}_{j,k},u^{n}_{j,k})$ can be rewritten as follows
	\begin{align}
		&|u^{n+1/2}_{j,k}|^2u^{n+1/2}_{j,k} -|u(x_j,y_k,t_{n+1/2})|^{2}u(x_j,y_k,t_{n+1/2})
		=\nonumber\\
		&\tau^2 u^{n+1/2}_{j,k}\bigg[\int_0^1s\int_0^1\Re(\bar{u}(x_j,y_k,t_{n+1/2})\partial_{tt}u(x_j,y_k,t_{n+1/2}-\frac{\tau s}{2}+s_1s\tau))ds_1ds\nonumber\\
		&+ \bigg|\frac{\tau}{2}\int_0^1s\int_0^1\partial_{tt}u(x_j,y_k,t_n+1/2-\frac{\tau s}{2}+s_1s\tau)ds_1ds\bigg|^2\bigg]\nonumber\\
		&+\frac{\tau^2}{2}|u(x_j,y_k,t_{n+1/2})|^{2}\int_0^1s\int_0^1\partial_{tt}u(x_j,y_k,t_{n+1/2}-\frac{\tau s}{2}+s_1s\tau)ds_1ds.\label{eq:mloss}
	\end{align}	
	Using $f=f_1$ with $f'(s) =1$, $f''(s)=0$ and $f=f_2$ with $f'(s)=2s$, $f''_2(s)=2$ for all $s\geq 0$, from equations \eqref{ineq:est_dt}, \eqref{ineq:est_dxx}, \eqref{ineq:est_dyy}, \eqref{eq:psi}, \eqref{eq:mloss}, the definition of $r^n_{j,k}$, and the first equation \eqref{eq:NLS} at the point $(x_j,y_k,t_{n+1/2})$, for $(j,k)\in \T_{JK}$, we have
	\begin{align*}
		|r^n_{j,k}|\lesssim h^2[\|\partial_{xxxx}u\|_{\infty} + \|\partial_{yyyy}u\|_{\infty} ]+  \tau^2[\tau^2\|u\|_{L^\infty}\|\partial_{tt}u\|^2_{\infty}+\|\partial_{ttt}u\|_{\infty}+\|\partial_{xxtt}u\|_{\infty} +\|\partial_{yytt}u\|_{L^\infty}   \\ 
		+(\|u\|^2_{\infty}+\|u\|^4_{\infty})\|\partial_{tt}u\|_{\infty} +(\|u\|_{\infty}+\|u\|^3_{\infty})\|\Re(\bar{u}\partial_{tt} u) +|\partial_t u|^2\|_{\infty} 
		+\|u\|_{L^\infty}\|u\partial_t u\|^2_{\infty} ].
	\end{align*}
	This completes the first estimation of \ref{lm:est_rn}.\\
	We note that $u=\partial_{tt} u=\partial_{xx}u=\partial_{yy} u= \partial_{xxxx}u=\partial_{yyyy} u =0$ on the boundary $\partial\Omega$. By performing the calculations similarly for $(j,k)\in \T_{JK}$, one can obtain
	\begin{align*}
		|\delta^+ r^n_{j,k}|\lesssim h^2 [\|\partial_{xxxx}\nabla u\|_{\infty} +\|\partial_{yyyy}\nabla u\|_{\infty}] + \tau^2[\|\partial_{ttt}\nabla u\|_{\infty}+\|\partial_{xxtt}\nabla u\|_{\infty} + \|\partial_{yytt}\nabla u\|_{\infty} \\ 
		+(\|u\nabla u\|_{\infty}+\|u^3\nabla u\|_{\infty})\|\partial_{tt} u\|_{\infty}+ (\|u\|^2_{\infty}+\|u\|^4_{\infty})\|\partial_{tt}\nabla u\|_{\infty}\\
		+\|\nabla u\|_{\infty}(1+\|u\|^2_{\infty})\|\Re(\bar{u}\partial_{tt} u) +|\partial_t u|^2\|_{\infty}+ (\|u\|_{\infty}+\|u\|^3_{\infty})\|\nabla(\Re(\bar{u}\partial_{tt} u) +|\partial_t u|^2)\|_{\infty}  \\
		+\|\nabla u\|_{\infty}\|u\partial_t u\|^2_{L^\infty} +\|u\|_{L^\infty}\|u\partial_t u\nabla(u\partial_t u)\|_{\infty}
		+\tau^2(\|\nabla u\|_{\infty}\|\partial_{tt}u\|^2_{\infty} +\| u\|_{\infty}\|\partial_{tt}u\partial_{tt}\nabla u\|_{\infty})] .
	\end{align*}
	This completes the second estimation of \ref{lm:est_rn}. 
\end{proof}

{}

\end{document}